\documentclass[12pt]{amsart}
\usepackage{fullpage,url,amssymb,amsmath,graphicx}

\usepackage{amsmath, amsthm, amssymb}
\usepackage{url}
\usepackage{graphicx}
\usepackage {amsmath}
\usepackage{amsthm}
\usepackage {amssymb}
\usepackage{float}

\usepackage{amsmath}
\usepackage {amssymb}
\usepackage {graphicx,enumerate,booktabs}
\usepackage {color, tikz}
\usepackage[all]{xy}

\newtheorem{definition}{Definition}
\newtheorem{lemma}{Lemma}[section]
\newtheorem{theorem}[lemma]{Theorem}

\newtheorem{cor}[lemma]{Corollary}

\newtheorem{claim*}{Claim}
\newtheorem{remark}[lemma]{Remark}

\newtheorem{example}[lemma]{Example}

\numberwithin{equation}{section}

\newcommand{\Hom}{\operatorname{Hom}}

\newcommand{\Sym}{\operatorname{Sym}}

\newcommand{\RR}{\mathbb{R}}
\newcommand{\HH}{\mathcal{H}}

\usepackage{tikz, tikz-3dplot}

\definecolor{cof}{RGB}{219,144,71}
\definecolor{pur}{RGB}{186,146,162}
\definecolor{greeo}{RGB}{91,173,69}
\definecolor{greet}{RGB}{52,111,72}

\begin{document}
\title{Semidefinite approximations of conical hulls of measured sets.}
\date{}
\author{Juli\'an Romero}
\address{Department of Combinatorics and Optimization \\
University of Waterloo\\
200 University Ave. W.\\ 
Waterloo, ON, N2L3G1, Canada. 
}
\email{jaromero@uwaterloo.ca}

\author{
Mauricio Velasco\\ 
}
\address{
Departamento de matem\'aticas\\
Universidad de los Andes\\
Carrera $1^{\rm ra}\#18A-12$\\ 
Bogot\'a, Colombia
}
\email{mvelasco@uniandes.edu.co}

\begin{abstract} Let $C$ be a proper convex cone generated by a compact set which supports a measure $\mu$. A construction due to A.Barvinok, E.Veomett and J.B. Lasserre produces, using $\mu$, a sequence $(P_k)_{k\in \mathbb{N}}$ of nested spectrahedral cones which contains the cone $C^*$ dual to $C$. We prove convergence results for such sequences of spectrahedra and provide tools for bounding the distance between $P_k$ and $C^*$. These tools are especially useful on cones with enough symmetries and allow us to determine bounds for several cones of interest. We compute such upper bounds for semidefinite approximations of cones over traveling salesman polytopes and for cones of nonnegative ternary sextics and quaternary quartics. 
\end{abstract}

\maketitle

\section{Introduction}
One of the main problems of convex optimization is the determination of the maximum value of a linear function over a convex set. Despite its reputation as the class of all ``tractable" optimization problems, such special cases contain an enormous variety of instances of very different complexity. Specifically, work of Gr\"{o}tschel, L\'{o}vasz and Schrijver~\cite{GLS} shows that the ability to approximate optima of linear programs on a class of convex sets in polynomial time to a prescribed accuracy is equivalent to the existence of a (weak) polynomial time membership oracle for this class. For many classes of convex sets, for instance the copositive cones~\cite{DG} or the cones of nonnegative polynomials~\cite{HL}, such membership problems have been shown to be NP hard. Nevertheless, optimizing over such complicated cones $C$ is often a problem of much interest. A possible alternative is to sacrifice precision for efficiency. We replace our convex set $C$ by a simpler convex set $C'$ which is a good approximation for $C$ in a suitable sense and such that efficient optimization over $C'$ is possible. Natural choices for such $C'$ are polyhedra and more generally spectrahedra or SDR sets. Approximations by such convex sets are of practical importance due to the availability of efficient interior point optimization algorithms~\cite{BN} on them.

The problem of how to construct such approximations $C'$ has been studied by several authors. There is considerable literature in the problem of  approximating convex bodies by polytopes (see for instance~\cite{Gruber} for a survey) as well as important work by Gouveia, Lasserre, Laurent, Parrilo, Thomas and others (see~\cite{P}, \cite{L3}, \cite{L2}, \cite{L4}) who propose approximation schemes for convex semialgebraic sets by SDR sets based on sums of squares relaxations.

In this article, we study another approximation strategy due to Barvinok and Veomett~\cite{BV} and Lasserre~\cite{L5}. Their construction allows us to approximate arbitrary convex sets which support a measure via a sequence of SDR sets. This article contains three main contributions concerning this approximation strategy. First, we give general conditions under which such sequences converge to the desired convex set. Second, we prove that in the presence of enough symmetries the speed of convergence (as measured by a scaling factor to be defined precisely)  can be determined by solving a semidefinite program. Third, we give explicit upper bounds for the scaling factors of these sequences on two special classes of convex sets: traveling salesman polytopes and cones of nonnegative polynomials in several variables. Such bounds are necessary if one intends to optimize over these cones following the approximation strategy outlined earlier and also have considerable mathematical interest since they are natural invariants of the cones.
 
In the remainder of this introduction we describe our results and the organization of the article in detail. We begin by defining some terminology. Let $V$ be a real finite-dimensional vector space and let $X$ be a compact topological space with a finite Borel measure $\mu$ supported on $X$ (i.e. such that the $\mu$-measure of every nonempty open set of $X$ is strictly positive). 
\begin{definition} A pair $(\phi,g)$ where $\phi: X\rightarrow V$ is a continuous function and $g:V\rightarrow \RR$ is a linear function is called admissible if the affine hull of $\phi(X)$ coincides with $g^{-1}(1)$. For an admissible pair $(\phi,g)$ we let $C:={\rm Cone}(\phi(X))$. \end{definition}
Note that $C$ is a proper cone with a distinguished interior point $\overline{x}:=\int_X \phi d\mu$.
\begin{definition} If $D\subseteq V$ is a cone then the dual cone $D^*$ is the set of elements $\ell\in V^*$ such that $\ell(d)\geq 0$ for every $d\in D$. The linear function $g$ is a distinguished interior point of the proper cone $C^*$ dual to $C$.
\end{definition}
The main objects of interest in this article are the convex cones $C$ and $C^*$ coming from admissible pairs. The following examples show that several interesting cones arise in this manner,
 
\begin{example}\label{Ex: guide} Fix a positive integer $n$,
\begin{enumerate}
\item{Let $X$ be the set of hamiltonian cycles in cities $1,\dots,n$ and let $\mu$ be the uniform measure. Let $\phi: X\rightarrow V\subseteq \Sym^2(\RR^n)$ 
be the map sending a hamiltonian cycle to its adjacency matrix and let $V$ be the subspace spanned by the images of all hamiltonian cycles. Let $g:V\rightarrow \RR$ be the map sending a matrix to $\frac{1}{2n}$ times the sum of its entries. In this case, $C$ is a cone over the Symmetric Traveling salesman polytope on the complete graph $K_n$.
}
\item{Let $d$ be a positive integer, let $X:=S^{n-1}$ be the unit sphere in $\RR^n$ and let $\mu$ be its normalized surface measure. Let $\phi: X\rightarrow \Sym^{2d}(\RR^n)$ be the map sending a point $v\in S^{n-1}$ to $v^{2d}$. Let $g: \Sym^{2d}(\RR^n)\rightarrow \RR$ be the unique linear map which sends $v^{2n}$ to $|v|^{2n}$. In this case $C^*$ is the cone of nonnegative homogeneous polynomials of degree $2d$ in $n$ variables. }
\item{Let $X$ be the intersection of the unit sphere $S^{n-1}\subseteq \RR^n$ and the non-negative orthant in $\RR^n$ and let $\mu$ be the restriction to $X$ of the normalized surface measure of the sphere.  Let $\phi: X\rightarrow \Sym^2(\RR^n)$ be the map sending $v$ to $v^2$. Let $g: \Sym^{2}(\RR^n)\rightarrow \RR$ be the unique linear map which sends $v^{2}$ to $|v|^{2}$. In this case $C^*$ is the cone of copositive quadratic forms.}
\end{enumerate}
\end{example}

As in the examples above, the exact determination of the cones $C$ and $C^*$ may be difficult. The following construction was introduced by Barvinok and Veomett~\cite{BV} (and is implicit in independent work by Lasserre~\cite{L3}) as a method to systematically construct approximations of $C^*$ and $C$ by spectrahedra  and SDR sets respectively. See also~\cite{V} for applications of this construction to multilinear optimization.
\begin{definition}\label{Def: BVL} Let $\mathcal{F}$ be a vector space of continuous, real valued functions on $X$. To $\lambda\in V^*$  we can associate a bilinear symmetric form $Q_\lambda: \mathcal{F}\times \mathcal{F}\rightarrow \RR$ via the formula
\[Q_\lambda(p,q)=\int_X \lambda(\phi(u)) p(u)q(u)d\mu(u).\]
Let $\Phi: V^*\rightarrow \Sym^2(\mathcal{F})^*$ be the linear map given by $\Phi(\lambda)=Q_{\lambda}.$
The BVL approximation of $C^*$ determined by $\mathcal{F}$, denoted $C^*(\mathcal{F})$, is the spectrahedral cone $\Phi^{-1}(S_+)$ where $S_+\subseteq  \Sym^2(\mathcal{F})^*$ is the cone of positive semidefinite quadratic forms on $\mathcal{F}$.
\end{definition}
It is immediate from the definition that the following statements hold,
\begin{enumerate}
\item{ For any $\mathcal{F}$ we have $C^*\subseteq C^*(\mathcal{F})$ and $C\supseteq C^*(\mathcal{F})^*$. We denote the SDR set $C^*(\mathcal{F})^*$ by $C(\mathcal{F})$.}
\item{ If $\mathcal{F}, \mathcal{G}$ are subspaces of real valued functions with $\mathcal{G}\subseteq \mathcal{F}$ then $C^*(\mathcal{G})\supseteq C^*(\mathcal{F})$ and $C(\mathcal{G})\subseteq C(\mathcal{F})$.}
\end{enumerate}
As the space of functions $\mathcal{F}$ becomes larger the spectrahedron $C^*(\mathcal{F})$ becomes smaller and the SDR set $C(\mathcal{F})$ larger. It is natural to ask whether by choosing sequences of vector spaces $\mathcal{F}_j$ appropriately we can make the sequences of cones $C^*(\mathcal{F}_j)$ and $C(\mathcal{F}_j)$ converge, in a suitable sense, to $C^*$ and to $C$. Our first result, proven in Section~$\S$\ref{Sec: Convergence}, shows that this happens under rather general hypotheses, generalizing~\cite[Lemma 3.1]{V}.
\begin{theorem} \label{Thm: Convergence} Assume $X$ is a compact Hausdorff topological space. Suppose that $(\mathcal{F}_j)_{j\in \mathbb{N}}$ are an increasing sequence of vector subspaces of the algebra of continuous functions on $X$ with the uniform norm. If $\bigcup_{j=1}^{\infty}\mathcal{F}_j$ is a subalgebra which separates points and contains the constant functions then the following equalities hold, 
\[
\begin{array}{ccc}
\bigcap_{j=1}^{\infty} C^*(\mathcal{F}_j)=C^* & \text{ and } & \overline{\bigcup_{j=1}^{\infty} C(\mathcal{F}_j)}=C.\\
\end{array}
\]
\end{theorem}
\begin{remark} If $\phi$ is one-to-one and $\mathcal{F}_j$ is the pullback of the homogenous polynomials of degree $j$ in $V$ to $X$ via $\phi$ then the vector spaces $\mathcal{F}_j$ satisfy the hypotheses of the Theorem. This is the original hierarchy studied by Veomett~\cite{V} for traveling salesman polytopes on the complete graph.
\end{remark}
Knowing that convergence does occur, the next step is to ask how quickly does convergence happen. To make this question meaningful it is necessary to have a quantitative measure of the ``distance'' between $C^*(\mathcal{F})$ and $C^*$ and between $C(\mathcal{F})$ and $C$ respectively. To define this quantity we need to fix bases for the relevant cones, 

\begin{definition} Let $\Lambda:=\{v\in V: g(v)=1\}$ and $\Lambda^{\vee}:=\{\ell\in V^*: \ell(\overline{x})=1\}$. Define $B:=C\cap \Lambda$, $B(\mathcal{F}):=C(\mathcal{F})\cap \Lambda$, $B^{\vee}:=C^*\cap \Lambda^{\vee}$ and $B^{\vee}(\mathcal{F}):=C^*(\mathcal{F})\cap \Lambda^{\vee}$. Note that $B^{\vee}(\mathcal{F})\supseteq B^{\vee}$ and that $B(\mathcal{F})\subseteq B$.
\end{definition}

\begin{definition}The scaling constant of $\mathcal{F}$, denoted $\alpha(\mathcal{F})$, is the infimum of the set of real numbers $\alpha$ such that the following two equivalent inclusions hold,
\[
\begin{array}{ccc}
\frac{1}{\alpha}\left(B^{\vee}(\mathcal{F})-g\right) \subseteq \left(B^{\vee}-g\right) & \text{ and } & \left(B-\overline{x}\right)\subseteq \alpha\left(B(\mathcal{F})-\overline{x}\right)
\end{array}
\]
 if this set is nonempty and equals infinity otherwise.
\end{definition}

Note that $\alpha(\mathcal{F})\geq 1$ and that $\alpha(\mathcal{F})=1$ if and only if the equalities $B^{\vee}(\mathcal{F})=B^\vee$ and $B(\mathcal{F})=B$ hold. Scaling constants are useful because they allow us to bound the maximum value of any linear function on the convex set $B^\vee$ (resp. on $B$) in terms of its maximum value on the spectrahedron $B^{\vee}(\mathcal{F})$ (resp. on the SDR set $B(\mathcal{F})$). More specifically, the following inequalities hold for any linear functions $f: V^*\rightarrow \RR$ and $h: V\rightarrow \RR$.
\[ \frac{1}{\alpha(\mathcal{F})}\left(\sup_{x\in B^{\vee}(\mathcal{F})}f(x)-f(g)\right)\leq \sup_{x\in B^{\vee}}f(x)-f(g) \leq \sup_{x\in B^{\vee}(\mathcal{F})}f(x)-f(g) \text{,}\]
\[ \sup_{x\in B(\mathcal{F})}h(x)-h(\overline{x})\leq \sup_{x\in B}h(x)-h(\overline{x}) \leq \alpha(\mathcal{F})\left(\sup_{x\in B(\mathcal{F})}h(x)-h(\overline{x})\right).\]

Our next result, proven in Section~$\S$\ref{Sec: Scaling}, shows that scaling constants are easily computable whenever $X$ has enough symmetries. To describe this concept precisely we need to introduce some additional terminology,

\begin{definition} \label{Def: enoughSymm} Let $H$ be a subgroup of the continuous automorphisms of $X$. We say that $H$ has enough symmetries if the following conditions hold:
\begin{enumerate}
\item{ The elements of $H$ preserve the measure $\mu$.}
\item{ The action of $H$ on $X$ is transitive.}
\item{ There exists a homomorphism $\Psi: H\rightarrow \Hom(V,V)$ such that for every $u\in X$ and $h\in H$ the equality $\Psi(h)(\phi(u))=\phi(h(u))$ holds.}
\item{ For every $f\in \mathcal{F}$ and $h\in H$ the function $f\circ h$ is an element of $\mathcal{F}$.}
\end{enumerate}
\end{definition}
The situation of having a subgroup $H$ with enough symmetries applies to the cones $(1)$ and $(2)$ in Example~\ref{Ex: guide}, but not to cone $(3)$.
\begin{theorem} \label{Thm: Scaling}The following statements hold:
\begin{enumerate}
\item{\label{Scaling} For any vector space $\mathcal{F}$ of real valued functions on $X$, the scaling constant is given by
\[ \alpha(\mathcal{F})=1-\inf_{\lambda \in B^{\vee}(\mathcal{F})}\inf_{x\in B} \lambda(x).\]
}
\item{\label{SymmScaling} Let $H$ be a subgroup of the continuous automorphisms of $X$.
\begin{enumerate}
\item{ If $H$ has enough symmetries then for any point $u_0\in X$ we have
\[\alpha(\mathcal{F})=1-\inf_{\lambda \in  B^{\vee}(\mathcal{F})}\lambda(\phi(u_0)).\]
}
\item{ If $H'\subseteq H$ is a compact subgroup which fixes $u_0$ and such that $\Psi_{|H'}$ is continuous then 
\[\alpha(\mathcal{F})=1-\inf_{\lambda \in  B^{\vee}(\mathcal{F})\cap W^{\vee}}\lambda(\phi(u_0)).\]
where $W^{\vee}\subseteq V^*$ is the subspace consiststing of linear forms $\ell$ such that $\ell\circ \Psi(h)=\ell$ for every $h\in H'$. 
}
\end{enumerate}
}
\end{enumerate}
\end{theorem}
In particular, if there exists a subgroup $H$ having enough symmetries then the scaling constants can be computed by semidefinite programming and under additional symmetries these programs can be simplified considerably.

Section~$\S$\ref{Sec: Applications} is devoted to applications of these ideas and contains the main results of the article. 

In Section~\ref{Sec: TSP} we study the scaling constants of BVL hierarchies on traveling salesman cones. Specifically, for a graph $G$ on $n$-vertices let $X$ be the set of hamiltonian cycles of $G$ endowed the uniform measure $\mu$ and let $\phi: X\rightarrow {\rm Sym}^2(\RR^n)$ be the map sending a cycle to its adjacency matrix. Define $V$ and $g$ as in Example~\ref{Ex: guide} $(1)$ and let $C:={\rm Cone}(\phi(X))$ be the traveling salesman cone of $G$. Letting $\mathcal{F}_k$ be the restriction, via $\phi$, of the homogeneous polynomials of degree $k$ in ${\rm Sym}^2(\RR^n)$, definition~\ref{Def: BVL} gives a hierarchy of SDR sets 
\[ C(\mathcal{F}_k)\subseteq C.\]
In Section~\ref{Sec: TSP} we develop a general framework for bounding the scaling constants $\alpha(\mathcal{F}_k)$ of this hierarchy for sufficiently symmetric graphs $G$. We then apply this framework to complete graphs and to complete bipartite graphs. Our main result is, 
\begin{theorem} \label{Thm: TSP} The following inequalities hold,
\begin{enumerate}
\item{ If $G$ is the complete graph on $n$ vertices and $k\in \{1,2,\dots , \lfloor\frac{n}{2}\rfloor\}$ then
\[
\alpha(\mathcal{F}_k)\leq 1+\frac{(n-1)}{2}\frac{\sum_{i=k}^{\lfloor \frac{n}{2} \rfloor\wedge 2k}{i \choose k}\beta_{i}}{\sum_{i=k}^{\lfloor \frac{n}{2} \rfloor\wedge 2k}{i \choose k}(\alpha_i-\beta_i)}\leq \frac{n}{k}+\frac{10}{n}
\]
where  $\alpha_i$ and $\beta_i$ are the explicit constants given in Equation (\ref{Eq:Pol_Veo_2}). 
}
\item{ If $G$ is the complete bipartite graph on two sets of $n$ vertices  and $k\in \{1,2,\dots,n\}$ then
\[
\alpha(\mathcal{F}_k)\leq 1+\frac{n}{2}\frac{\sum_{i=k}^{n\wedge 2k}{i \choose k}\eta_{i}}{\sum_{i=k}^{n\wedge 2k}{i \choose k}(\gamma_i-\eta_i)}\leq \frac{2 n}{k}+\frac{2 (k+1)}{k (2n-k-3)}
\]
where $\gamma_k$ and $\eta_k$ are the explicit constants given by equations (\ref{Eq:Bipar1}) and (\ref{Eq:Bipar2}).
}
\end{enumerate}
\end{theorem}
\begin{remark}Theorem~\ref{Thm: TSP} part $(1)$ improves on earlier upper bounds by Veomett~\cite{BVstsp}.\end{remark}

In Section~$\S$\ref{Sec: P} we study the scaling constants for BVL hierarchies on cones of nonnegative polynomials. We fix positive integers $n$ and $d$ and let $X:=S^{n-1}$ be the unit sphere in $\RR^n$ with normalized surface measure $\mu$ and define $g$ as in Example~\ref{Ex: guide} $(2)$. The cone $C^*\subseteq \Sym^{2d}(\RR^n)^*$ consists of nonnegative polynomials of degree $2d$ in $n$ variables. Let $\mathcal{F}_k$ be the vector space of homogeneous polynomials of degree $2k$ in $n$ variables restricted to $X$. Such spaces specify a sequence of spectrahedra 
\[C^*(\mathcal{F}_k)\supseteq C^*\]
and we study its basic properties. In Lemma~\ref{Lemma: basicNonnegative} we give explicit formulas for the matrices defining the spectrahedra $C^*(\mathcal{F}_k)$ and prove that these spectrahedra converge to $C^*$. 

By a Theorem of Hilbert, the cone $C^*$ equals the cone of sums of squares of forms of degree $d$ if and only if either $n\leq 2$ or $d\leq 1$ or $(n,d)=(3,2)$ (see~\cite{BJAMS} for a modern proof). It follows that in these cases the cone $C^*$ has a simple spectrahedral description. In all other cases $C^*$ is not a spectrahedron and it is interesting to use the spectrahedral approximations $C^*(\mathcal{F}_k)$. The simplest cases of interest are thus ternary sextics $(n,d)=(3,3)$ and quaternary quartics $(n,d)=(4,2)$. Our first result result gives (numerical) upper bounds for the scaling constants $\alpha(\mathcal{F}_k)$,

\begin{theorem} \label{Thm: NumBounds}The following tables contain numerically computed upper bounds for the scaling constants $\alpha(\mathcal{F}_k)$ (see Figure~\ref{boundsQS}).
\begin{enumerate}
\item{for quaternary quartics $(n,d)=(4,2)$
\[
\begin{tiny}
\begin{array}{|c|c|c|c|c|c|c|c|c|c|c|c|}
\hline
k & \alpha(\mathcal{F}_k)\leq & k & \alpha(\mathcal{F}_k)\leq & k & \alpha(\mathcal{F}_k)\leq  & k & \alpha(\mathcal{F}_k)\leq & k & \alpha(\mathcal{F}_k)\leq & k & \alpha(\mathcal{F}_k)\leq  \\
\hline
2 & 2.755044 & 3 & 1.949091 & 4 & 1.607291 & 5 & 1.425842 & 6 & 1.316627 & 7 & 1.245305\\ 
8 & 1.195964 & 9 & 1.160319 & 10 & 1.133685 & 11 & 1.113236 & 12 & 1.097181 & 13 & 1.084338\\ 
14 & 1.073898 & 15 & 1.065294 & 16 & 1.058117 & 17 & 1.052067 & 18 & 1.046919 & 19 & 1.042501\\ 
\hline
\end{array}
\end{tiny}
\]
}
\item{and for ternary sextics $(n,d)=(3,3)$
\[
\begin{tiny}
\begin{array}{|c|c|c|c|c|c|c|c|c|c|c|c|}
\hline
k & \alpha(\mathcal{F}_k)\leq & k & \alpha(\mathcal{F}_k)\leq & k & \alpha(\mathcal{F}_k)\leq  & k & \alpha(\mathcal{F}_k)\leq & k & \alpha(\mathcal{F}_k)\leq & k & \alpha(\mathcal{F}_k)\leq  \\
\hline
3 & 2.668980 & 4 & 2.055565 & 5 & 1.746356 & 6 & 1.526781 & 7 & 1.387487 & 8 & 1.315492\\ 
9 & 1.259626 & 10 & 1.213191 & 11 & 1.176739 & 12 & 1.154268 & 13 & 1.134184 & 14 & 1.116618\\ 
15 & 1.102053 & 16 & 1.091873 & 17 & 1.082356 & 18 & 1.073801 & 19 & 1.066682 & 20 & 1.061063\\ 
\hline
\end{array}
\end{tiny}
.\]

}
\end{enumerate}
\end{theorem}

\begin{figure}[h]
\label{boundsQS}
\includegraphics[scale=0.7]{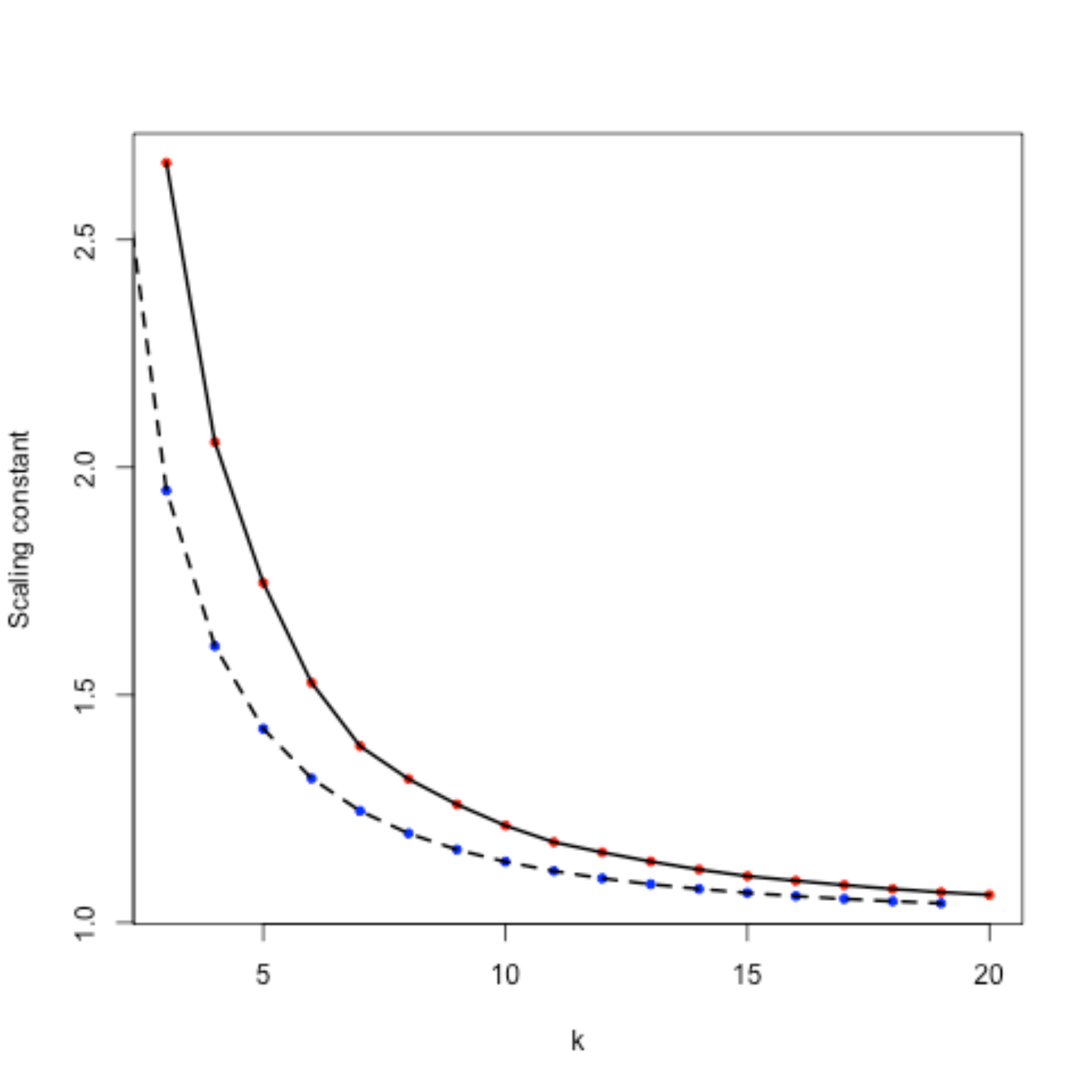}
\caption{Upper bounds for $\alpha(\mathcal{F}_k)$ for quaternary quartics (dashed line) and ternary sextics (continuous line). }
\end{figure}

Next we derive a method for constructing upper bounds for $\alpha(\mathcal{F}_k)$  for any $(n,d)$ with $k\geq d$. To describe the method we need to introduce some additional notation. Let $h_0:=1$ and for $m\geq 1$ let $h_m:=\binom{n+m-1}{n-1}-\binom{n+m-3}{n-1}$ and for an integer $j>0$ let $L_j(x)$ be the Legendre polynomial of degree $j$. The appearance of Legendre polynomials, the key element for the upper bounds obtained in this section, is a natural consequence of symmetry (see Section~\ref{Sec: P} for details).

\begin{theorem} \label{Thm: PSDbound} For $k\geq d$ let $R_k$ be the set of polynomials of the form 
\[q(x_1)=\sum_{j=0}^{d}h_{2j}L_{2j}(x_1)+\sum_{j=d+1}^{k} b_j L_{2j}(x_1)\text{ with $b_j\in \mathbb{R}$}\] 
where $L_j$ denotes the Legendre polynomial of degree $j$. The following inequality holds,
\[\alpha(\mathcal{F}_k)\leq 1-\sup_{q\in R_k}\min_{x_1\in [-1,1]} q(x_1).\]
In particular, any polynomial $q \in R_k$ gives an upper bound
\[ \alpha(\mathcal{F}_k)\leq 1-\min_{x_1\in [-1,1]} q(x_1).\] 
\end{theorem}
As a result, we obtain the following upper bound for any $(n,d)$ and $k=d$,
\begin{cor} \label{cor: PSD1} The following inequality holds, 
\[\alpha(\mathcal{F}_d)\leq 1+\frac{1}{(1-\gamma^2)^{\frac{1}{4}}}\left( \sum_{j=0}^d h_{2j} \sqrt{\frac{4}{\pi(4j+1)}}\right)\]
where $\gamma$ is the biggest root of the polynomial $L_{2d}(x_1)$. The root $\gamma$ is known to satisfy
\[\gamma=\cos\left(\frac{\beta}{\sqrt{4d^2+2d+1/3}}\left(1-\frac{\beta^2-2}{360(4d^2+2d+1/3)^2}\right)\right)+O(d^{-7})\]
where $\beta\approx 2.4048$ is the first positive root of the Bessel function $J_0(x)$.
\end{cor}

\subsection*{Acknowledgements}
We wish to thank Grigoriy Blekherman for several insightful conversations during the completion of this project. M. Velasco was partially supported by the FAPA funds from Universidad de los Andes.

\section{Convergence results for BVL hierarchies.}\label{Sec: Convergence}
In this section we prove the results about convergence of BVL hierarchies and scaling constants described in the introduction. Theorem~\ref{Thm: ScalingDual} gives a dual point of view on scaling constants which will be useful in Section~\ref{Sec: Applications}.

\begin{proof}[Proof of Theorem \ref{Thm: Convergence}] Since for every index $j$ the inclusion $C^*\subseteq C^*(\mathcal{F}_j)$  holds we have $C^*\subseteq \bigcap C^*(\mathcal{F}_j)$. If $\lambda\not\in C^*$ then there exists a point $c\in C$ such that $\lambda(c)<0$ and thus there exists $u\in X$ such that $\lambda(\phi(u))=\epsilon<0$. Let $U:=\{x\in X: \lambda(\phi(u))<\frac{\epsilon}{2}\}.$ The set $U$ is open and nonempty. Since $X$ is a normal topological space, the set $U$ contains a nonempty open set $A$ such that $A\subseteq \overline{A}\subseteq U$ and By Urysohn's Lemma there exists a continuous function $p: X\rightarrow [0,1]$ such that $p(\overline{A})=1$ and $p(X\setminus U)=0$. It follows that 
\[ \int_X \lambda(\phi(u))p(u)^2d\mu(u)=:\beta \leq \epsilon\mu(A)<0\] 
Now, $\bigcup \mathcal{F}_j$ is an algebra which separates points and contains the constants and $X$ is a compact Hausdorff topological space. Thus, by the Stone-Weierstrass Theorem the algebra $\bigcup \mathcal{F}_j$ it is dense in the algebra of continuous functions on $X$ with the uniform norm. In particular there exists a sequence of functions $p_k\in \mathcal{F}_{j_k}$ which converges uniformly to $p$. As a result
\[ \lim_{k\rightarrow \infty} \int_X \lambda(\phi(u))p_k^2(u)d\mu(u)=\beta<0\]  
and thus there is an index $j_s$ such that $\lambda\not\in C^*(\mathcal{F}_{js})$ proving the claimed equality. The equality $\overline{\bigcup_{j=1}^{\infty} C(\mathcal{F}_j)}=C$ follows immediately from bi-duality for proper cones.
\end{proof}
\begin{remark} The previous Theorem extends~\cite[Theorem 3.2]{L3} where $\phi$ is assumed to be scalar valued and $\mathcal{F}_k$ are assumed to be homogeneous polynomials of degree $k$. It also extends~\cite[Lemma 3.1]{V} where $\phi$ is assumed to be a polynomial map.
\end{remark}
\begin{remark} The key argument in the proof of the above Lemma lies in trying to approximate the Dirac $\delta_p$ measure centered at a point  $p$ with sums of squares of elements of $\mathcal{F}$. It follows from the above argument that if $\delta$ can be represented exactly by a sum of squares of elements of $\mathcal{F}$ then the equality  $C^*=C^*(\mathcal{F})$ holds. As an interesting consequence, if $X$ is finite and $\mathcal{F}_k$ is the set of polynomials of degree $k$ restricted to $X$ via $\phi$ then $C^*=C^*(\mathcal{F}_k)$ for some integer $k$ generalizing~\cite[Section 1.3]{BV}. This occurs because every function on a finite set is represented by a polynomial. The integer $k$ is bounded above by the Castelnuovo-Mumford regularity of the variety $\phi(X)$ (see~\cite[Section 20.5]{Eisenbud}).
\end{remark}

\begin{remark} One of the advantages of the setup of this article is that the spaces of fuctions $\mathcal{F}_j$ are now intrinsic to $X$ and do not depend on the function $\phi$. In particular a sequence of subspaces satisfying the hypothesis in the above Theorem can be used for constructing a converging sequence of approximations for the cone induced by any admissible pair $(\phi,g)$. 
\end{remark}

\subsection{Scaling constants}\label{Sec: Scaling}
\begin{proof}[Proof of Theorem \ref{Thm: Scaling}] (\ref{Scaling}) For a positive real number $s$, the inclusion $\frac{1}{s}(B^{\vee}(\mathcal{F})-g)\subseteq \left(B^{\vee}-g\right)$ holds if and only if for every $\lambda\in B^{\vee}(\mathcal{F})$ we have
\[\frac{1}{s}(\lambda-g)+g=\frac{1}{s}(\lambda+(s-1)g)\in B^{\vee}.\]
This condition holds iff the linear function $\beta:=\frac{1}{s}(\lambda+(s-1)g)\in C^*$ because $\beta(\overline{x})=1$. This occurs for every $\lambda\in B^{\vee}(\mathcal{F})$ if and only if $s-1\geq -\inf_{\lambda \in B^{\vee}(\mathcal{F})}\inf_{x\in B} \lambda(x)$. It follows that $\alpha(\mathcal{F})$ is given by the formula above. 
(\ref{SymmScaling}a) Since $H$ has enough symmetries there is a linear representation $\Psi$ of $H$ on $V$. This representation induces a linear action of $H$ on $V^*$ via $\Psi^*(h)(\lambda)(x)=\lambda(\Psi(h)(x))$. We claim that if $H$ has enough symmetries then for every $h\in H$ and $\lambda\in B^{\vee}(\mathcal{F})$ we have $\psi^*(h)(\lambda)\in B^{\vee}(\mathcal{F})$. This is because for any $p\in \mathcal{F}$ we have
\[\int_X \Psi^*(h)(\lambda(\phi(u))p(u)^2d\mu(u) = \int_X \lambda(\Psi(h)(\phi(u)))p(u)^2d\mu(u)=\int_X \lambda(\phi(h(u)))p(u)^2d\mu\]
where the second equality holds by Definition~\ref{Def: enoughSymm}, property $(3)$. Now, by Definition~\ref{Def: enoughSymm} property $(1)$, the last integral equals
\[\int_X \lambda(\phi(y))p(h^{-1}(y))^2d\mu(y)\geq 0\]
which is nonnegative because $\lambda\in B^{\vee}(\mathcal{F})$ and because $p\circ h^{-1}$ is an element of $\mathcal{F}$ by Definition~\ref{Def: enoughSymm} property $(4)$. It follows that $\Psi^*(h)(\lambda)\in C^*(\mathcal{F})$. Moreover $\Psi^*(h)(\lambda(\overline{x}))=\lambda(\Psi(h)(\overline{x}))=\lambda(\overline{x})=1$ showing that $\Psi^*(h)(\lambda)\in \Lambda^{\vee}$ proving the claim. Finally if $u_0$ is any point of $X$ and $\lambda\in B^{\vee}(\mathcal{F})$ achieves its minimum on $B$ at a point $\phi(u')$ for some $u'\in X$ then, by Definition~\ref{Def: enoughSymm} property $(2)$, there exists $h\in H$ such that $h(u_0)=u'$. Now, $\Psi^*(h)(\lambda)(\phi(u_0))=\lambda(\phi(h(u_0)))=\lambda(\phi(u'))$ and since $\Psi^*(h)(\lambda)\in B^{\vee}(\mathcal{F})$ we have 
\[\inf_{\lambda \in B^{\vee}(\mathcal{F})}\inf_{x\in B} \lambda(x)=\inf_{\lambda \in  B^{\vee}(\mathcal{F})}\lambda(\phi(u_0))\]
obtaining the claimed formula for the scaling constant.
$(\ref{SymmScaling}b)$ Assume $H'\subseteq H$ is a locally compact subgroup which fixes $u_0$ and such that $\Psi_{|H'}$ is continuous. Let $\eta\in B^{\vee}(\mathcal{F})$ be such that 
\[\eta(\phi(u_0))= \inf_{\lambda \in  B^{\vee}(\mathcal{F})}\lambda(\phi(u_0)).\]
Let $\nu$ be the Haar probability measure on the compact group $H'$~\cite{Haar} and define the linear form 
\[\tau(x):=\int_{H'} \Psi^*(h)(\eta(x))d\nu(h).\] Note that $\tau(\phi(u_0))=\eta(\phi(u_0))$ because $H'$ fixes $u_0$. Also $\Psi^*(h)(\tau)=\tau$ since the Haar-measure on $H'$ is left invariant and finally $\tau\in B^{\vee}(\mathcal{F})$ because, by the previous paragraph, it is an expected value of elements of $B^{\vee}(\mathcal{F})$. As a result, if  $W^{\vee}\subseteq V^*$ denotes the subspace of linear forms $\ell$ such that $\Psi^*(h)(\ell)=\ell$ for every $h\in H'$ then the following equality holds
\[\inf_{\lambda \in  B^{\vee}(\mathcal{F})}\lambda(\phi(u_0))=\inf_{\lambda \in  B^{\vee}(\mathcal{F})\cap W}\lambda(\phi(u_0)).\]
and we conclude that the scaling constant can be computed from the right hand side as claimed.
\end{proof}
\begin{remark} Theorem~\ref{Thm: Scaling} Part $(1)$ generalizes the proof of~\cite[Theorem 1.1]{BV}.
\end{remark}

The dual point of view presented in the following Theorem is often useful for the determination of scaling constants.

\begin{definition}\label{Def: Sigmay} For $y\in B$ let $\Sigma(y)$ be the set of elements $s(u)$ which are sums of squares of elements of $\mathcal{F}$ and satisfy
\[ y-\int_X \phi(u)s(u)d\mu(u)\in {\rm span}(\overline{x}).\]
Note that $\phi$ is vector-valued and thus the integral is a vector in $V$. 
\end{definition}
\begin{theorem} \label{Thm: ScalingDual} 
The following statements hold,
\begin{enumerate}
\item{\label{Scaling} The scaling constant of $\mathcal{F}$ is given by
\[ \alpha(\mathcal{F})=\sup_{y\in B} \inf_{r\in \Sigma(y)}\int_X r(u)d\mu(u).\]
}
\item{\label{SymmScaling} If $H$ is a subgroup of the continuous automorphisms of $X$ which has enough symmetries then for any point $u_0\in X$ we have
\[\alpha(\mathcal{F})=\inf_{r\in \Sigma(\phi(u_0))}\int_X r(u)d\mu(u).\]
}
\end{enumerate}
\end{theorem}
\begin{proof} We will prove that for any $y\in B$ the equality 
\[\inf_{\lambda\in B^{\vee}(\mathcal{F})}\lambda(y)=\sup_{r\in \Sigma(y)} \left(1-\int_X r(u)d\mu(u)\right)\]
holds. Once this claim is established, both parts follow immediately from Theorem~\ref{Thm: Scaling}. To establish the claim note that $\inf_{\lambda\in B^{\vee}(\mathcal{F})}\lambda(y)$ is a semidefinite optimization problem which we will refer to as the primal problem. Its Lagrangian dual is given by
\[\sup_{(\nu,r)} \nu \text{: $y-\nu\overline{x}=\int_X\phi(u)r(u)d\mu$, $r\in \Sigma(y).$}\]
applying $g$ to the linear constraint we obtain 
\[1-\nu=g(y)-\nu g(\overline{x})=\int_X g(\phi(u))r(u)d\mu=\int_X r(u)d\mu(u)\]
and thus the dual problem is equivalent to $\sup_{r\in \Sigma(y)} \left(1-\int_X r(u)d\mu(u)\right)$.
Moreover, since $X$ is compact there exists a linear form $\lambda\in V^*$ which is strictly positive on $\phi(X)$ and satisfies $\lambda(\overline{x})=1$. Since the measure $\mu$ is supported in all of $X$ it follows that the quadratic form $Q_{\lambda}$ is strictly positive definite. As a result the primal semidefinite problem is strictly feasible and thus strong duality holds proving the claimed equality.

\end{proof}

\section{Applications}\label{Sec: Applications}
This Section contains the main results of the article. These are upper bounds for the scaling constants of BVL approximations for traveling salesman cones and for certain cones of nonnegative polynomials.
\subsection{Traveling salesman cones.}\label{Sec: TSP}
In this section we study BVL approximations of the cone over the symmetric traveling salesman polytope of an undirected hamiltonian graph $G$. We denote this cone by $STS(G)$. The cone $STS(G)$ is generated by the adjacency matrices of hamiltonian cycles on the graph $G$. Our main results are upper bounds on the scaling constants of the BVL approximation of $STS(G)$ defined by restrictions of homogeneous polynomials when $G$ is the complete graph $K_{n}$ or the complete bipartite graph $K_{n,n}$. Our upper bounds on the scaling constants for $STS(K_n)$ improve those given by Veomett in~\cite{BVstsp}. Our results on $STS(K_{n,n})$ give the first known upper bounds.

We begin by describing the cones $STS(G)$ in the setting of admissible pairs. Fix an undirected hamiltonian graph $G$ with vertices labeled $1,\dots,n$. Let $E(G)$ be its edge set, consisting of pairs $ij$ with $1\leq i<j\leq n$ for which vertex $i$ is adjacent to vertex $j$. Let $X$ be the set of hamiltonian cycles in $G$ and let $\mu$ be the uniform measure on $X$. Let $\phi: X\rightarrow \Sym^2(\RR^n)$ be the map sending a cycle to its adjacency matrix and let $V:={\rm im}(\phi)$ be the subspace spanned by the images of all cycles. Let $g:V\rightarrow \RR$ be the linear map which sends a matrix to $\frac{1}{2n}$ times the sum of its entries. Note that $C:={\rm Cone}(\phi(X))$ equals the cone $STS(G)$.

The space $\Sym^2(\RR^n)$ has a natural basis given by the $(0,1)$-matrices $e_{ij}$ in which all but the entries $ij$ and $ji$ are equal to zero. We denote its dual basis by $\hat{x}_{ij}$ and define the ring $R:=\RR[\hat{x}_{ij}: ij\in E(G)]$. Composition with $\phi$ allows us to restrict an element of the ring $R$ to a function on $X$. We denote the restriction of monomials by $x_{ij}:=\hat{x}_{ij}\circ \phi$.  The elements of $\phi(X)$ are $(0,1)$ matrices and thus the equality $x_{ij}^2=x_{ij}$ holds for every $ij\in E(G)$. It follows that for every finite set $M$ of multi-indices $\alpha=(\alpha_{ij})_{ij\in E(G)}$ and real numbers $(a_{\alpha})_{\alpha\in M}$ we have the equality
\[ \sum_{\alpha \in M} a_{\alpha} x^{\alpha} = \sum_{\alpha\in M} a_{\alpha} x^{\beta(\alpha)}\]
where $(\beta(\cdot))_{ij}$ is the support function given by
\[
\beta(\alpha)_{ij}:=
\begin{cases}
1,\text{ if $\alpha_{ij}>0$ and}\\
0,\text{ otherwise. }
\end{cases}
\]
For a multi-index $\alpha$ we let ${\rm supp}(\alpha)$ be the support graph of $\alpha$. This is the subgraph of $G$ whose adjacency matrix is given by $\beta(\alpha)$.

\begin{lemma} \label{TSPSimple} Let $\mathcal{F}_k$ be the vector space of  homogeneous polynomials of degree $k$ in $R$ restricted to $X$. The following statements hold,
\begin{enumerate}
\item{The pair $(\phi,g)$ is admissible.}
\item{ For any multi-index $(\alpha_{ij})_{ij\in E(G)}$ we have
\[\int_X x^{\alpha}d\mu=\int_X x^{\beta(\alpha)}d\mu=\frac{|\{ \text{ hamiltonian cycles $u\supseteq {\rm supp}(\alpha)$}\}|}{|\text{ hamiltonian cycles of $G$}|}\]
In particular the above integral is zero whenever ${\rm supp}(\alpha)$ is not either a hamiltonian cycle or a union of vertex-disjoint paths and isolated vertices. 
} 
\item{Let $H={\rm Aut}(G)$ be the group of graph automorphisms of $G$. The following statements hold,
\begin{enumerate}
\item{If $H$ acts transitively on $E(G)$ then $\hat{x}_{ij}(\overline{x})$ is independent of $ij\in E(G)$.}
\item{If $H$ acts transitively on the set of hamiltonian cycles of $G$ then $H$ has enough symmetries.} 
\end{enumerate}
}
\end{enumerate}
\end{lemma}
\begin{proof} $(1)$ and the first equality in $(2)$ are immediate from the paragraph preceding the Lemma. The value of a monomial $x^{\alpha}$ on a hamiltonian cycle $u$ is either $1/h$ where $h$ is the number of hamiltonian cycles in $G$ or zero depending on whether the cycle $u$ contains the support graph ${\rm supp}(\alpha)$. The second equality in $(2)$ follows from this. $(3b)$ The group $H$ acts bijectively on $X$ and thus preserves the uniform measure. The action of $g\in H$ on $\Sym^2(\RR^n)$ is determined by sending $e_{ij}$ to $e_{g(i)g(j)}$ for $1\leq i<j\leq n$. This action is compatible with the action of $H$ on hamiltonian cycles.  The induced action of $H$ on the elements of $R$ is given by linear changes of coordinates and thus preserves degree and maps the set $\mathcal{F}_k$ to itself. Our assumption guarantees that condition $(2)$ of Definition~\ref{Def: enoughSymm} is satisfied. $(3a)$ The point $\overline{x}$ is fixed by $H$ since $H$ acts by permuting the hamiltonian cycles on $G$. $(3b)$ Transitivity of the action of $H$ on $X$ gives the only remaining requirement for $H$ to have enough symmetries in the sense of Definition~\ref{Def: enoughSymm}.
\end{proof}

The following Lemma provides a useful tool for computing upper bounds for scaling constants of BVL approximations of the cone $STS(G)$ using the vector spaces $\mathcal{F}_k$. Recall that the set $\Sigma(y)$ was introduced in Definition~\ref{Def: Sigmay}. 

\begin{lemma} \label{TwoValues} Assume that $H$ acts transitively on $E(G)$. Fix $u_0\in X$ and let $s\in \mathcal{F}_{2k}$ be a sum of squares of elements of $\mathcal{F}_k$. Then $s\in \Sigma(\phi(u_0))$ if and only if there exist real numbers $\alpha,\beta$ with $\alpha-\beta=1$ such that, for every $ij\in E(G)$ we have
\[ \int_X x_{ij}sd\mu = 
\begin{cases}
\alpha,\text{ if $ij\in u_0$ and}\\
\beta,\text{ if $ij\not\in u_0$.}
\end{cases}\]
\end{lemma}
\begin{proof} If $s\in \Sigma(\phi(u_0))$ then there is a real number $\gamma$ such that
\[\int_X \phi sd\mu=\phi(u_0)+\gamma\overline{x}\]
Evaluating the linear functional $\hat{x}_{ij}$ for $ij\in E(G)$ on both sides we obtain
\[\int_X x_{ij} sd\mu= x_{ij}(u_0) + \gamma\hat{x}_{ij}(\overline{x})\]
Since $H$ acts transitively on $E(G)$ the value $\hat{x}_{ij}(\overline{x})=\eta$ is independent of $ij\in E(G)$. It follows that the integral on the left hand side assumes only the values $\alpha:=1+\gamma\eta$, if $ij\in u_0$ and $\beta:=\gamma\eta$ if $ij\not\in u_0$ as claimed so $\alpha-\beta=1$. Conversely, for any $s$ satifying the above hypotheses we have the equality
\[ \int_X\phi sd\mu-\frac{\beta}{\eta}\overline{x}=(\alpha-\beta)\phi(u_0)=\phi(u_0)\]
and thus $s\in \Sigma(\phi(u_0))$ as claimed.
\end{proof}

\begin{cor}\label{CorTSP} Fix $u_0\in X$. If $H$ has enough symmetries and $\hat{x}_{ij}(\overline{x})=\eta$ for all $ij\in E(G)$ then
\[\alpha(\mathcal{F}_k)=1+ \inf_{(\alpha,\beta)} \frac{\beta}{\eta(\alpha-\beta)} \]
Where $\alpha$ and $\beta$ range over all pairs of real numbers $\alpha>\beta$ such that there exists a sum of squares $s\in \mathcal{F}_{2k}$ such that for every $ij\in E(G)$ we have
 \[ \int_X x_{ij}sd\mu = 
\begin{cases}
\alpha,\text{ if $ij\in u_0$ and}\\
\beta,\text{ if $ij\not\in u_0$.}
\end{cases}\]
\end{cor}
\begin{proof}  Since $H$ has enough symmetries we know by Theorem~\ref{Thm: ScalingDual} part $(2)$ that 
\[\alpha(\mathcal{F}_k)=\inf_{s\in \Sigma(\phi(u_0))}\int_X sd\mu.\]
By Lemma~\ref{TwoValues}, $\tau\in \Sigma(\phi(u_0))$ if and only if there exist $\alpha>\beta$ with $\alpha-\beta=1$ such that the following equality holds, 
\[ \int_X\phi \tau d\mu-\frac{\beta}{\eta}\overline{x}=\phi(u_0)\]
Applying the linear function $g$ on both sides we obtain the equality
\[ \int_X \tau d\mu =1+\frac{\beta}{\eta}.\]
If $s$ is any sum of squares satisfying the hypotheses above then letting $\tau:=\frac{1}{\alpha-\beta}s$ we see that $\tau\in \Sigma(\phi(u_0))$ obtaining the claimed bound.
\end{proof}

\begin{remark} Corollary~\ref{CorTSP} generalizes~\cite[Lemma 2]{BV} to arbitrary graphs. Moreover, we allow a sum of squares $s$ rather than only sums of  $\{0,1\}$-valued polynomials in $\phi(X)$. This extension is the main reason why we will be able to improve the scaling constants for $STS(K_n)$ found in~\cite{BVstsp}.
\end{remark}

Next we want to use Corollary~\ref{CorTSP} when $G$ is either the complete graph $K_n$ or the complete bipartite graph $K_{n,n}$. To this end we need to understand the integrals over $X$ of the monomials in $\mathcal{F}_k$. By Lemma~\ref{TSPSimple} we can restrict our attention to monomials whose support is either a disjoint union of non-closed paths and isolated vertices or a single hamiltonian cycle.

\begin{lemma}\label{Lem:PathDecom}
Let $x^{\alpha}$ be a monomial in $\mathcal{F}_k$. The following statements hold:
\begin{enumerate}
\item \cite[Lemma 2]{BV} Let $G=K_n$. If  ${\rm supp}(\alpha)$ is the union of $l\geq 1$ vertex disjoint non-overlapping paths $p_1,...,p_l$ and isolated vertices then
\begin{equation}
\int_X x^{\alpha}\ d \mu=2^{l}\frac{(n-p-1)!}{(n-1)!},
\end{equation}
where $p$ is the sum of the lengths of such paths.
\item Let $G=K_{n,n}$.  If ${\rm supp}(\alpha)$ is the union of $l\geq1$  vertex disjoint non-overlapping paths $p_1,...,p_l$ and isolated vertices then 
\begin{equation}
\int_{X}x^{\alpha}\ d\mu=2^{l-r+1}\frac{(2n-p-1)!}{n!(n-1)!}{2n-p-r \choose \frac{2n-p-r}{2}}^{-1},
\end{equation}
where $p$ is the sum of the lengths of $p_1,\dots, p_l$ and $r\leq l$ is the number of these paths which have odd length. 
\end{enumerate}
\end{lemma}
\begin{proof} $(2)$ Let $V=V_1\cup V_2$ with $|V_1|=|V_2|=n$ be the bipartition of the vertex set of $K_{n,n}$.  Assume first that $l>r$. We will count the hamiltonian cycles containing the support graph of $\alpha$ by constructing sequences which represent the order in which such cycles traverse the connected components of ${\rm supp}(\alpha)$. For $i=1,2$ let $S_i$ be the set consisting of the vertices in $V_i$ which are isolated in ${\rm supp}(\alpha)$ and the paths (of even length) in ${\rm supp}(\alpha)$ with endpoints in $V_i$. Note that $|S_1|=|S_2|=\frac{2n-p-r}{2}>0$. Fix a path $s_1^1\in S_1$ with endpoints $v_1$ and $v_2$. Starting with the path $s_1^1$ construct a sequence  $s_1^1, s_1^2,s_2^1,\dots, s_{|S_1|}^1, s_{|S_2|}^2$ of distinct elements $s_i^j\in S_j$. Suppose $q_1,\dots, q_r$ are the paths of odd length in ${\rm supp}(\alpha)$. We place the $q_i$ in any order always to the right of some $s_i^j$, generating a a new sequence, say
$$
s_1^1\dashrightarrow q_{i_1}\dashrightarrow\dots \dashrightarrow q_{j_1}\dashrightarrow s_1^2\dashrightarrow\cdots \dashrightarrow s_{|S_2|}^2\dashrightarrow q_{i_{r}} \dashrightarrow \cdots \dashrightarrow q_{j_r}.
$$
Joining each pair of consecutive paths or vertices in this sequence with a single edge we will produce a hamiltonian cycle in $K_{n,n}$. Each path that belongs to some $S_i$ has two possible ways to be connected with the next element of the sequence by edges of $K_{n,n}$, if this element is a path of even length, and only one if it is an isolated vertex.  On the other hand, every path $q_i$ has its endpoints on different $V_i$ and hence can be connected by an edge of $K_{n,n}$ in only one way to the next element of the sequence. 
We conclude that, if we force the cycle to start with the path $s_1^1$ traversed from $v_1$ towards $v_2$ then the total number of hamiltonian cycles which contain the support of $\alpha$ is equal to,
$$
(\#\text{ Sequences $s_i^j$})(\#\text{ Placements of the $q_i$s })(\#\text{Ways to connect them})=
$$ 
$$
((|S_1|-1)!(|S_2|)!)\left(\frac{(2|S_1|+r-1)!}{(2|S_1|-1)!}\right)2^{l-r-1}=2^{l-r}(2n-p-1)!{2|S_1| \choose |S_1|}^{-1}.
$$
Which yields the above formula after division by $|X|=\frac{n!(n-1)!}{2}$.
The case $l-r=0, r>0$ is addressed similarly, starting with a fixed oriented path $q_1$ of odd length. Finally, if $\alpha=0$ then the formula holds trivially.\end{proof}

\begin{example}[Veomett Polynomials] Let $u_0$ be a hamiltonian cycle of $K_n$ and let $M$ be the set of all maximum size matchings of $K_n$ that are unions of edges of $u_0$. Every element of $M$ has exactly $\lfloor \frac{n}{2} \rfloor$ edges and hence the cardinality of $M$ is $2$ or $n$ depending on whether $n$ is even or odd. For each $k=1,2,\dots,\frac{n}{2}$ define the polynomials 
\begin{equation}
s_{k}(x)=\sum_{\Gamma\in M}\sum_{\substack{L\subset \Gamma\\L=k}}x^{L},
\end{equation}
where $x^{L}$ denotes the product of all $x_{ij}$ with $\{i,j\}\in L$. Veomett shows in~\cite{BVstsp} that these polynomials are elements of $\Sigma(\phi(u_0))$ and that
\begin{equation}\label{Eq:Pol_Veo_2}
\int_{X}x_{ij}s_k \ d\mu= \begin{cases}
\alpha_{k}  &\text{if } ij\in u_0,\\ \beta_k &\text{if } ij\notin u_0. 
\end{cases}
\end{equation}
where $\alpha_k$ and $\beta_k$ are defined by  
\begin{align*}
\alpha_k&=2^{k-2}\left(4 (k+2) r^2-2 (k (k+7)+4) r+3 k (k+3)\right) \frac{(r-2)!
   (2 r-k-2)!}{k!(r-k)!}, \\
 \beta_k&=2^{k-1} \left(\binom{r-2}{k-2}+4 \binom{r-1}{k}\right)  (2 r-k-2)!.
\end{align*}
if $n=2r$ and by
\begin{align*}
\alpha_k&=2^{k-2} \left(4 (k+2) r^2-2 k (k+4) r+(k-1) k+4 r\right)\frac{ (r-1)! (2r-k-1)!}{k! (r-k)!}	, \\
 \beta_k&=2^{k-2} \left(k^2 (2 r-1)-8 k r^2+6 k r+k+4 (r-1) r (2 r+1)\right) \frac{(r-2)!
   (2 r-k-1)!}{k!(r-k)!}.
\end{align*}
if If $n=2r+1$. As a result~\cite[Main Theorem]{V} the following inequality holds,
\begin{equation}
\alpha(\mathcal{F}_k)\leq1+\frac{(n-1)}{2}\frac{\beta_k}{\alpha_k-\beta_k}\leq \frac{n}{k}+\frac{10}{n}.
\end{equation} 
\end{example}

The next Theorem is an improvement of the above bounds,

\begin{theorem}\label{Thm:NewKn}  Define  $\alpha_i$ and $\beta_i$ as in (\ref{Eq:Pol_Veo_2}). For $k=1,2,\dots, \lfloor \frac{n}{2} \rfloor$ and $G=K_n$ the following inequality holds,
\begin{equation}\label{Eq:RV_Pol_1}
\alpha(\mathcal{F}_k)\leq 1+\frac{(n-1)}{2}\frac{\sum_{i=k}^{\lfloor \frac{n}{2} \rfloor\wedge 2k}{i \choose k}\beta_{i}}{\sum_{i=k}^{\lfloor \frac{n}{2} \rfloor\wedge 2k}{i \choose k}(\alpha_i-\beta_i)}\leq \frac{n}{k}+\frac{10}{n}.
\end{equation}

\end{theorem}
\begin{proof}
Let $u_0$ be a hamiltonian cycle of $K_n$ and let $M$ be the set of all maximum size matchings of $K_n$ that are the union of edges of $u_0$. For $k\leq \lfloor \frac{n}{2} \rfloor$ define the polynomials
$$
\hat{s}_k(x)=\sum_{\Gamma\in M}\left(\sum_{\substack{L\subset \Gamma\\|L|=k}}x^L\right)^2.
$$
Then, $\hat{s}_k$ is again a sum of squares of elements of $\mathcal{F}_k$ and coincides with the sum
$$
\hat{s}_k(x)=\sum_{i=k}^{\lfloor \frac{n}{2} \rfloor\wedge 2k}{i \choose k}s_{i}(x)\in\Sigma(\phi(u_0)),
$$
hence proving the first inequality on (\ref{Eq:RV_Pol_1}). The second inequality holds because  
$$
\frac{n-1}{2}\frac{\beta_i}{\alpha_i-\beta_i}\leq \frac{n}{k}+\frac{10}{n} -1,
$$
for any $i\geq k$.
\end{proof}

\begin{remark}
The expressions for the above bounds are rather complicated. In the case $k=1$, using computer-aided simplification, these expressions simplify to 
$$
\alpha(\mathcal{F}_1)\leq\frac{2 n}{3}+ \frac{4 n}{3 \left(3 n^2-15 n+16\right)}
$$ 
when $n$ is even and 
$$
\alpha(\mathcal{F}_1)\leq\frac{2 n}{3}-\frac{2}{5 (n-2)}+\frac{154}{45 (3 n-11)}+\frac{1}{9}
$$
when $n$ is odd. From extensive computer calculations, we conjecture the ``improvement ratio" between the bounds given by Theorem~\ref{Thm:NewKn} and those in~\cite{V} satisfies the following inequality
\begin{equation}\label{Eq:Ratio}
\frac{1+\frac{(n-1)}{2}\cdot\frac{\sum_{i=k}^{\lfloor \frac{n}{2} \rfloor\wedge 2k}{i \choose k}\beta_{i}}{\sum_{i=k}^{\lfloor \frac{n}{2} \rfloor\wedge 2k}{i \choose k}(\alpha_i-\beta_i)}}{1+\frac{(n-1)}{2}\frac{\beta_k}{\alpha_k-\beta_k}}\leq 1+\left(1-\frac{2k}{n}\right)\log\left(\frac{k+2}{k+3}\right).
\end{equation}

\begin{figure}[h]
\centering
\begin{minipage}[b]{0.45\linewidth}
 \includegraphics[scale=0.5]{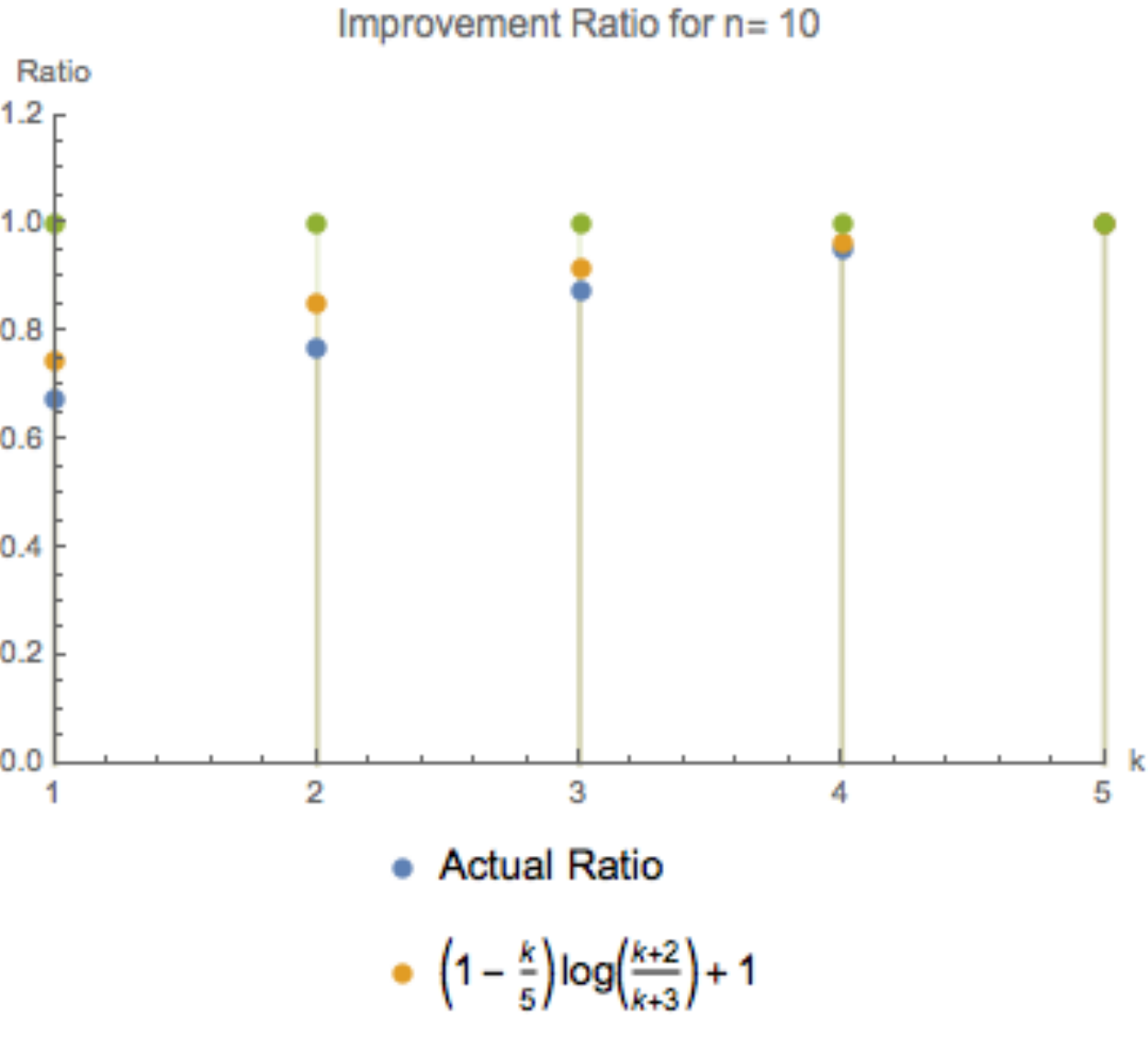}
  \label{fig:minipage1}
\end{minipage}
\quad
\begin{minipage}[b]{0.45\linewidth}
\includegraphics[scale=0.5]{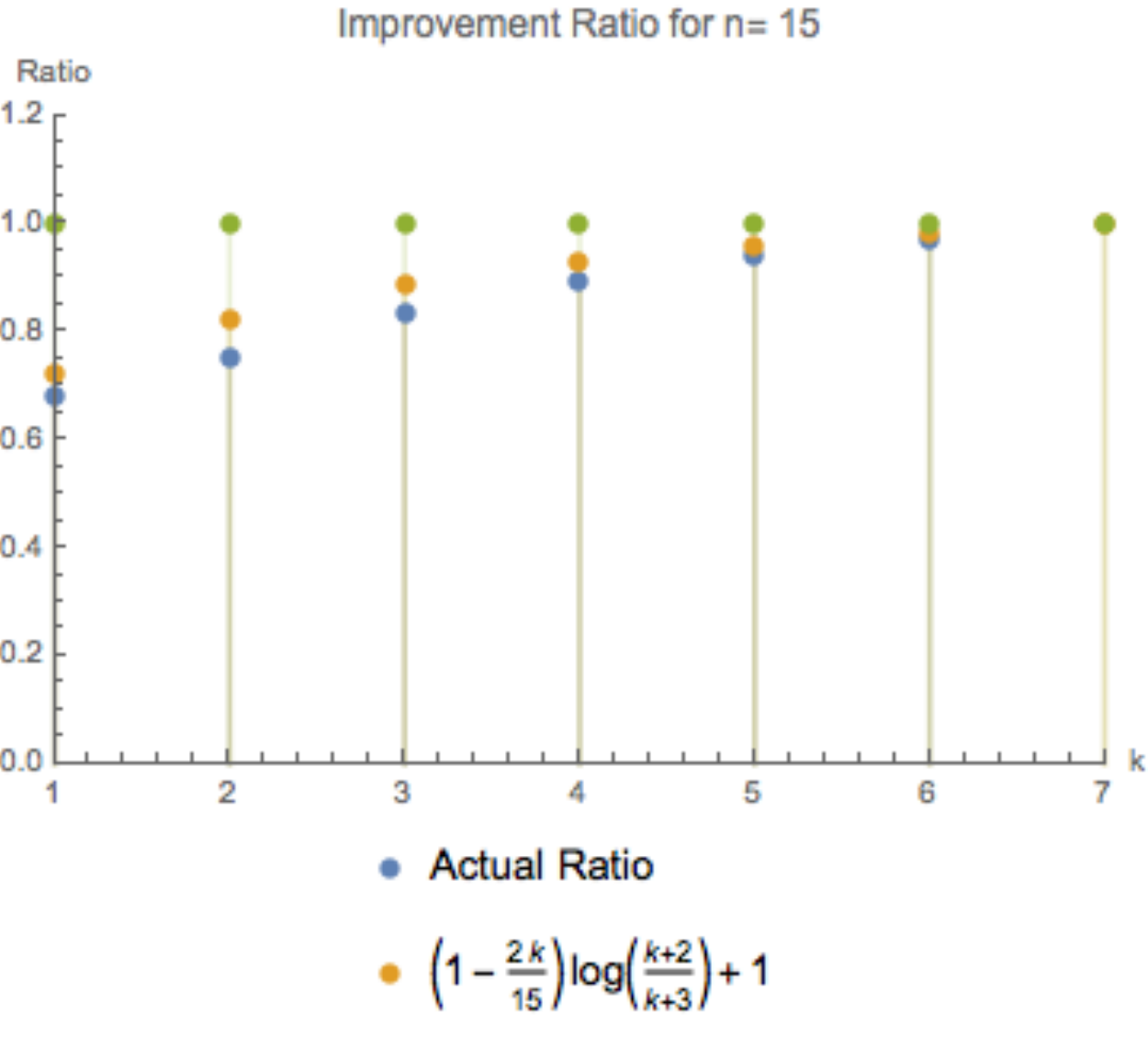}
  \label{fig:minipage2}
\end{minipage}
\end{figure}
\begin{figure}[H]
\centering
\begin{minipage}[b]{0.45\linewidth}
 \includegraphics[scale=0.5]{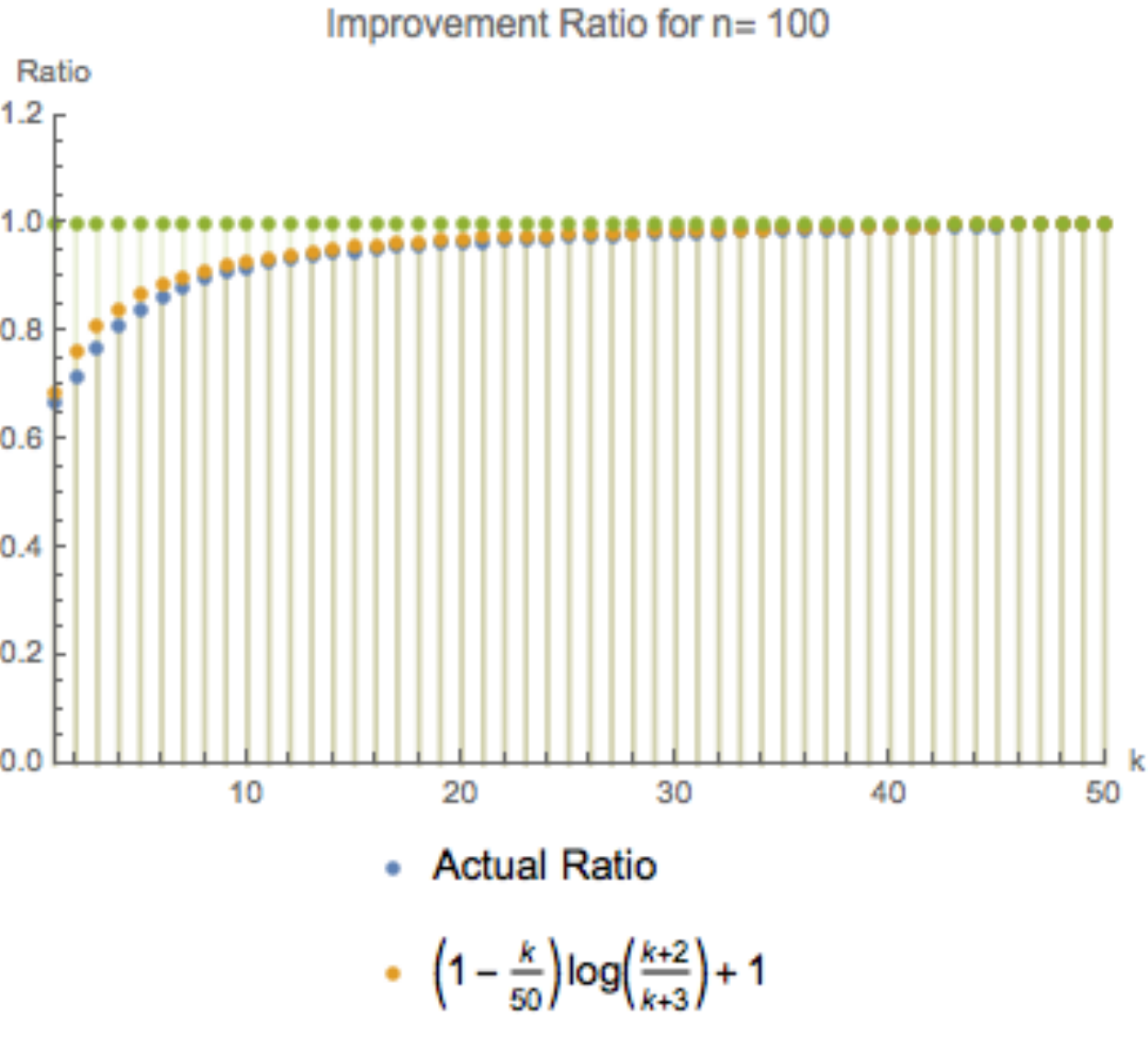}
  \label{fig:minipage1}
\end{minipage}
\quad
\begin{minipage}[b]{0.45\linewidth}
\includegraphics[scale=0.5]{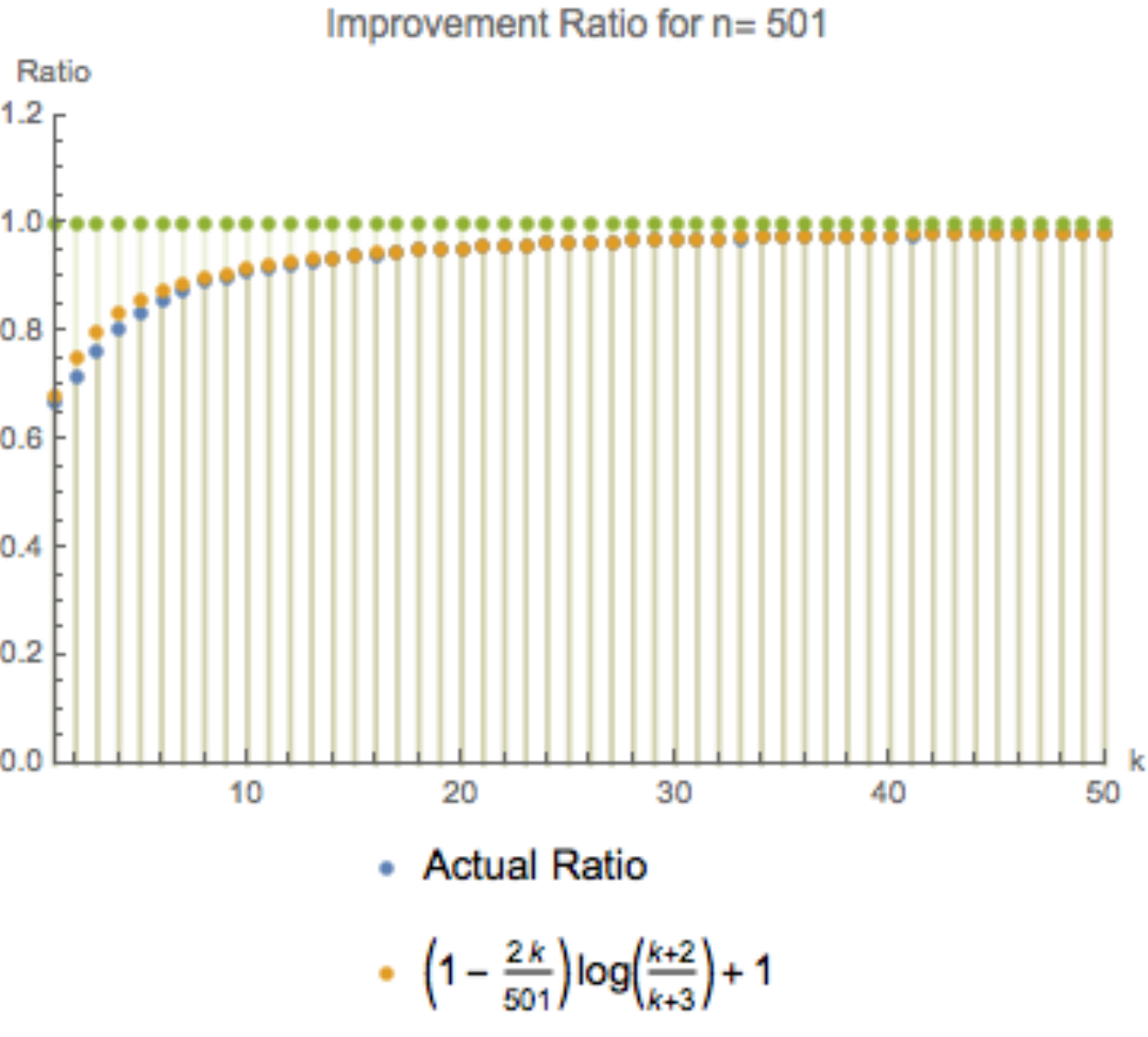}
  \label{fig:minipage2}
\end{minipage}
\end{figure}
\begin{figure}[H]
\includegraphics[scale=0.5]{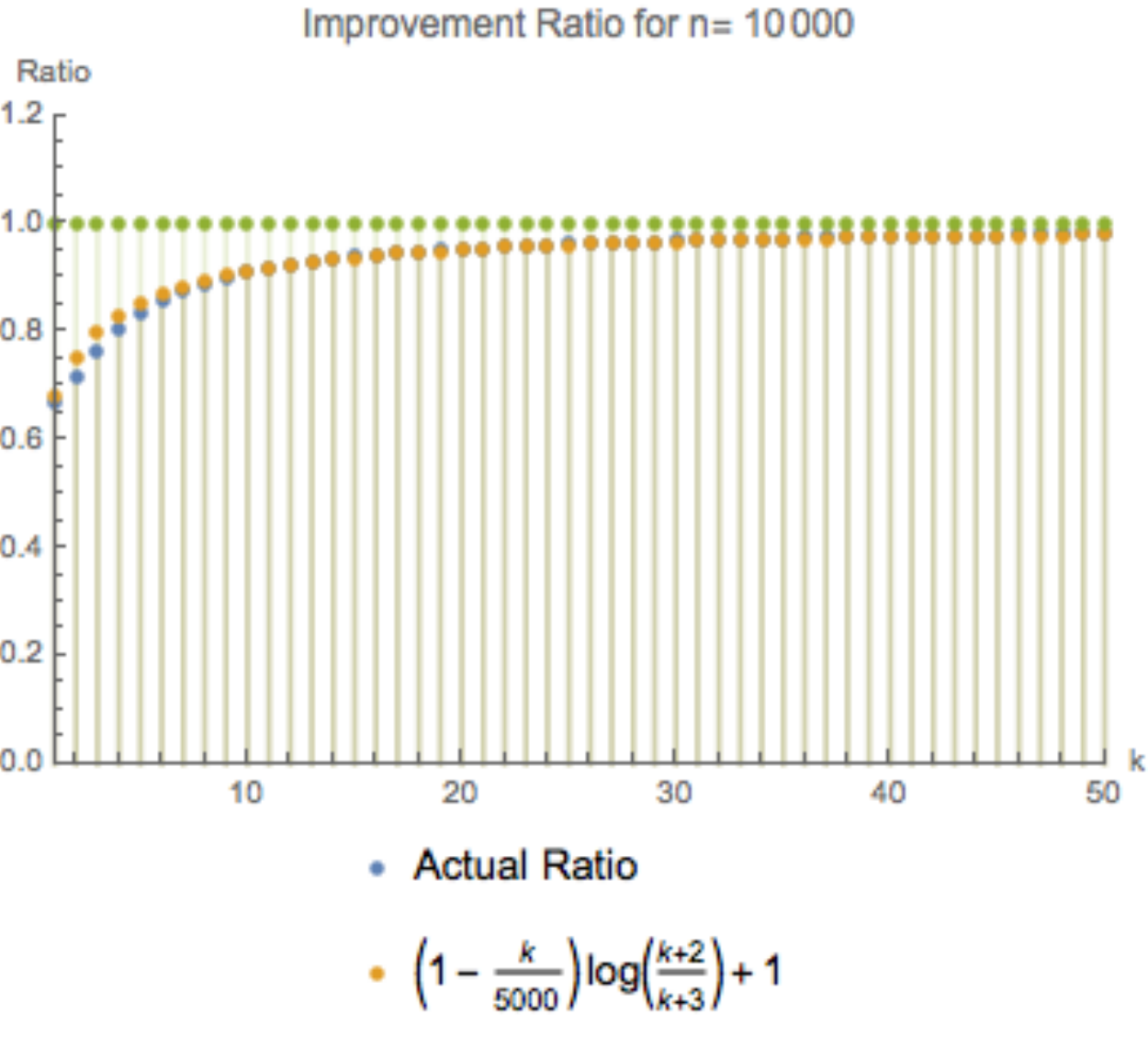}
\caption{The ratio in (\ref{Eq:Ratio}) for $n=10,15,100,501,10000$ and $k=1,\dots, 50$.}
\end{figure}
\end{remark}

Next, we focus on the case of $STS(G)$ when $G$ is the bipartite graph $K_{n,n}$

\begin{theorem}\label{Thm:Knn}
Let $G=K_{n,n}$ then for $k\in\{1,2,\dots,n\}$
$$
\alpha(\mathcal{F}_k)\leq \frac{2 n}{k}+\frac{2 (k+1)}{k (2n-k-3)}.
$$ 
\end{theorem}

\begin{proof}
Let $u_0$ a hamiltonian cycle in $K_{n,n}$. This hamiltonian cycle can be partitioned into two perfect matchings $\Gamma_1$ and $\Gamma_2$. For $k=1,\dots, n$ define the polynomials
$$
p_k(x):=\sum_{\substack{L\subset \Gamma_1\\|L|=k}}x^L+\sum_{\substack{L\subset \Gamma_2\\|L|=k}}x^L,
$$

We claim that the integrals $\int_X p_kx_{ij}d\mu$ assume only two values depending on whether or not $ij\in u_0$. It follows that $p_i$ will give us to bound on $\alpha(\mathcal{F}_k)$ via Corollary~\ref{CorTSP}.

If $\{i,j\}\in \Gamma_1$ then the collection of all $L\subset \Gamma_1$ with $|L|=k$ can be partitioned into two sub-classes: The ones that contain the edge $\{i,j\}$ and the ones that do not. If $L$ lies in the first class, then the support graph of $x_{ij}x^L$ will be the union of $k$ disjoint non-overlapping paths of length $1$. If not, then such graph will be the union of $k+1$ disjoint non-overlapping paths of length $1$. Similarly, we can divide the collection of all $L\subset \Gamma_2$ with $|L|=k$ into three sub-classes: the subgraphs that do not touch the edge $\{i,j\}$, the ones that touch one endpoint of $\{i,j\}$ and the ones that touch both endpoints of $\{i,j\}$. Computing the types of supporting graphs that these three classes generate and using Lemma \ref{Lem:PathDecom}, we obtain:
\small{
$$
\gamma_k:=\frac{n!(n-1)!}{2}\int_{X}x_{ij}p_k \ d\mu= {n-1 \choose k-1}(2n-k-1)!{2n-2k \choose n-k}^{-1} + 
$$ 
\begin{equation}\label{Eq:Bipar1}
+ {n-1 \choose k}(2n-k-2)!{2(n-k-1) \choose n-k-1}^{-1}+{n-2 \choose k}(2n-k-2)!{2(n-k-1) \choose n-k-1}^{-1}+
\end{equation}

$$
+2^2{n-2 \choose k-1}(2n-k-2)!{2n-2k \choose n-k}^{-1}+{n-2 \choose k-2}(2n-k-2)!{2n-2k \choose n-k}^{-1}.
$$}
\normalsize
If $\{i,j\}$ does not belong to either of the matchings $\Gamma_i$ then this case is similar to the case $\{i,j\}\in\Gamma_1\setminus \Gamma_2$ and thus
$$
\eta_k:=\frac{n!(n-1)!}{2}\int_{X}x_{ij}p_k \ d\mu=2\left({n-2 \choose k}(2n-k-2)!{2(n-k-1) \choose n-k-1}^{-1}+
\right.$$
\begin{equation}\label{Eq:Bipar2}
\left. +2^2{n-2 \choose k-1}(2n-k-2)!{2n-2k \choose n-k}^{-1}+{n-2 \choose k-2}(2n-k-2)!{2n-2k \choose n-k}^{-1}\right).
\end{equation}
Note that $\gamma_k$ and $\eta_k$ dependent of $ij$ only on whether or not the edge belongs to the hamiltonian cycle $u_0$.
The claimed bound follows from Corollary \ref{CorTSP}.
\end{proof}

\begin{remark}\label{Rm:Knn}
Define the polynomials
$$
\hat{p_k}(x):=\left(\sum_{\substack{L\subset \Gamma_1\\|L|=k}}x^L\right)^2+\left(\sum_{\substack{L\subset \Gamma_2\\|L|=k}}x^L\right)^2=\sum_{i=k}^{n\wedge 2k}{i \choose k}p_{i}(x).
$$
The above bounds can be improved arguing as in Theorem~\ref{Thm:NewKn} yielding
$$
\alpha(\mathcal{F}_k)\leq 1+\frac{n}{2}\frac{\sum_{i=k}^{n\wedge 2k}{i \choose k}\eta_{i}}{\sum_{i=k}^{n\wedge 2k}{i \choose k}(\gamma_i-\eta_i)}\leq \frac{2 n}{k}+\frac{2 (k+1)}{k (2n-k-3)},
$$
where $\gamma_k$ and $\eta_k$ are given by equations (\ref{Eq:Bipar1}) and (\ref{Eq:Bipar2}).
\end{remark}

\begin{proof}[Proof of Theorem \ref{Thm: TSP}]  Follows immediately from Theorem~\ref{Thm:NewKn}, Theorem~\ref{Thm:Knn} and Remark~\ref{Rm:Knn}.
\end{proof}

\subsection{Nonnegative polynomials}\label{Sec: P}
Fix positive integers $n$ and $d$. In this section we study BVL approximations of the cone of homogeneous nonnegative polynomials of degree $2d$ in $n$ variables. Such approximations were considered by Lasserre in~\cite{L3} and our contribution are upper bounds on the scaling constants. The reader is referred to~\cite{KM} for basic properties of symmetric powers of vector spaces.

We begin by describing these cones in the setting of admissible pairs. Let $V:=\Sym^{2d}(\RR^n)$ and let $X:=S^{n-1}$ be the unit sphere in $\RR^{n}$ with the normalized surface measure $\mu$. Define $\phi: X\rightarrow V$ by sending a vector $v$ in the unit sphere to the element $v^{2d}\in V$. The dual space $V^*:=\Sym^{2d}(\RR^{n})^*$ is naturally identified with the space of homogeneous polynomials of degree $2d$ in $n$ variables. Under this identification the natural bilinear pairing $\Sym^{2d}(\mathbb{R}^{n})\times \Sym^{2d}(\RR^{n})^*\rightarrow \RR$ is $\langle v^{2d},p\rangle := p(v)$, the value at $v$ of the homogeneous polynomial $p$. Define $g: \Sym^{2d}(\RR^n)\rightarrow \RR$ be the unique linear map which sends an element of the form $v^{2d}$ to $|v|^{2d}$. $g$ is the linear map corresponding to the polynomial $(x_1^2+\dots+x_n^2)^{d}$ under the pairing.

\begin{lemma}\label{Lemma: basicNonnegative}  Let $\mathcal{F}_k$ be the vector space of homogeneous polynomials of degree $2k$ in $n$ variables restricted to $X$. The following statements hold:
\begin{enumerate}
\item{ The pair $(\phi,g)$ is admissible. Moreover $C^*\subseteq \Sym^{2d}(\RR^{n})^*$ is the cone of nonnegative polynomials of degree $2d$ in $n$ variables.}
\item{  The spectrahedral cone $C^*(\mathcal{F}_k)$ can be described explicitly. It is given by the homogeneous polynomials $p$ of degree $2d$ in $n$ variables which satisfy
\[C^*(\mathcal{F}_k)=\left\{ p=\sum a_{\alpha}x^{\alpha}:  \sum a_{\alpha} A_{\alpha}\succeq 0\right\}\]
where the symmetric matrix $A_{\alpha}$ has rows and columns indexed by multiindices of degree $2k$ and entries given by
\[(A_{\alpha})_{\beta_1,\beta_2}:=\int_{X}x^{\alpha+\beta_1+\beta_2}d\mu=
\begin{cases}
0,\text{ if some $\alpha_i$ is odd}\\
\frac{2\Gamma(\eta_1)\dots\Gamma(\eta_n)}{\Gamma(\eta_1+\dots+\eta_n)}
\end{cases}
\]
where $x^{\alpha+\beta_1+\beta_2}=\prod x_i^{\gamma_i}$ and $\eta_i:=\frac{\gamma_i+1}{2}$.
}

\item{ For any polynomial $p$ in $\Sym^{2d}(\RR^n)^*$  we have $\langle p,\overline{x}\rangle =\int_{S^{n-1}} p(u)d\mu(u)$.}
\item{Let $H:=O(n)$ be the group of orthogonal transformations of $\RR^{n}$. The group $H$ has enough symmetries.}
\end{enumerate}
\end{lemma}
\begin{proof} $(1)$ Since $X$ is the unit sphere the identity $g\circ \phi=1$ holds. In coordinates the map $\phi$ sends the vector $(v_1,\dots, v_n)$ to the vector whose components are indexed by the monomials of degree $2d$ in $n$ variables $e_1^{a_1}\dots e_n^{a_n}$ with $a_1+\dots +a_n=2d$, $a_i\geq 0$ and with corresponding coefficients $\binom{2d}{a_1,\dots,a_d}v_1^{a_1}\dots v_n^{a_n}$. Since the coefficients form a basis for the space of homogeneous polynomials of degree $2d$ in the variables $v_1,\dots, v_n$, the set $\phi(X)$ is not contained in any proper affine subspace of $g^{-1}(1)$. 
$(2)$ This is obtained by using the formula of Folland~\cite{Folland} for computing integrals of monomials over spheres.
$(3)$ By definition the equality $\overline{x}=\int_{S^{n-1}} (v_1e_1+\dots v_ne_n)^{2d}d\mu$ holds. As a result, 
\[ \langle p,\overline{x}\rangle =\int_{S^{n-1}} \langle p, (v_1e_1+\dots v_ne_n)^{2d}\rangle d\mu = \int_{S^{n-1}} p(v_1,\dots,v_n)d\mu.\]
$(4)$ The action of $O(n)$ on $S^{n-1}$ is obviously measure-preserving and transitive. Any representation $W$ of a group induces new representations via its symmetric powers. In the case of the natural representation of $O(n)$ on $\RR^n$ the induced action on $\Sym^{2d}(\RR^n)$ by an element $A\in O(n)$ is given by the unique linear map which satisfies $A\cdot(v^{2d}):=(A\cdot v)^{2d}$ for every $v$. This equality proves property $(3)$ in Definition~\ref{Def: enoughSymm}. Finally, a homogeneous linear change of coordinates maps homogeneous polynomials of a given degree to homogeneous polynomials of the same degree and thus elements of $\mathcal{F}_j$ are mapped to $\mathcal{F}_j$ by elements of $H$. The claim follows.
\end{proof}

The main reason we will be able to bound the scaling constants is the fact that the stabilizer in $O(n)$ of a point of the sphere is a sufficiently large group. We will denote the stabilizer of the point $e_1$ with $S(e_1)\cong O(n-1)$. To take advantage of this fact we need to recall some basic facts about harmonic polynomials on the sphere $S^{n-1}$.

\begin{definition} Let $\Delta:=\sum_{i=1}^n \frac{\partial^2}{\partial x_i^2}$ be the Laplacian operator in $\RR^n$. A homogeneous polynomial $f\in \RR[x_1,\dots,x_n]$ is called harmonic if $\Delta f\equiv 0$. For an integer $m\geq 0$ let $\mathcal{H}_{m}$ denote the vector space of harmonic polynomials of degree $m$. 
\end{definition}
Recall~\cite[Theorem 5.7]{HFT} that any homogeneous polynomial $q$ of degree $m$ in $\RR^n$ can be written uniquely as
\[ q= p_m + s p_{m-2}+s^2p_{m-4}+\dots +s^{k}p_{m-2k}\]
where $s=x_1^2+\dots+x_n^2$, $k:=\lfloor\frac{m}{2}\rfloor$ and $p_j\in \mathcal{H}_j$.
As a result, for $m\geq 2$ we have the equality
\[\dim \mathcal{H}_m=\binom{n+m-1}{n-1}-\binom{n+m-3}{n-1}.\]
Moreover, the spaces $\mathcal{H}_m$ are orthogonal with respect to the inner product 
\[\langle f_1,f_2\rangle :=\int_{S^{n-1}} f_1(u)f_2(u)d\mu(u).\]

\begin{definition} Let $v\in S^{n-1}$ and let $Z_m(x, v)$ be the unique element of $\mathcal{H}_m$ such that for every $f\in \mathcal{H}_m$
\[ \int_{S^{n-1}} f(u)Z_m(u, v)d\mu(u)= f(v).\]
We denote by $Z_m(x)$ the special case when $v=e_1$ and call it the zonal harmonic of degree $m$. 
\end{definition}
The zonal harmonic of degree $m$ can be characterized as the unique polynomial which satisfies the following properties~\cite[Theorem 2.3]{BC},
\begin{itemize}
\item{ $Z_m(x)\in \HH_m$}
\item{$Z_m(e_1)=\dim \HH_m$}
\item{$Z_m(A\cdot u)=Z_m(u)$ for every orthogonal matrix $A$ such that $A(e_1)=e_1$.}
\end{itemize}
It follows that for any $(u_1,\dots, u_n)$ in $S^{n-1}$ the equality $Z_m(u)=h_m L_{m}(u_1)$ holds where $L_m(x_1):=\frac{1}{m!2^m}\frac{d^m}{dx_1^m}[(x_1^2-1)^m]$ is the Legendre polynomial of degree $m$ and $h_m:=\dim \HH_m$. 

\begin{proof}[Proof of Theorem~\ref{Thm: NumBounds}] By Theorem~\ref{Thm: Scaling} part $(2)$ we know that
\[ \alpha(\mathcal{F}_k)=1-\inf_{\lambda \in W^{\vee}\cap B^{\vee}(\mathcal{F}_k)} \lambda(e_1)\]
so we want the set of forms $\lambda\in \Sym^{2d}(\RR^n)^*$ which are invariant under $S(e_1)$, have an average value of one on the sphere and for which 
$\int_{X} \lambda(u)q(u)^2d\mu\geq 0$ for every homogeneous polynomial $q$ of degree $2k$. Since $\lambda$ is $S(e_1)$-invariant we can assume that $\lambda$ is a linear combination of zonal polynomials of even degree
\[ \lambda=a_0r^dZ_0(x)+\dots + a_dZ_{2d}(x).\]
By definition of Zonal polynomial the integrals $\int_X Z_{2t}(x)q_i(x)q_j(x)d\mu$ can be computed by evaluating the harmonic component of  degree $2t$ of $q_i(x)q_j(x)$ at the north pole. The harmonic components of a polynomial can be computed by using~\cite[Theorem 5.21]{HFT}. This allows us to set up a semidefinite program for the exact computation of $\alpha(\mathcal{F}_k)$. However the dimensions of the matrices in this program grow very quickly with $k$. We observe that symmetry suggests a canonical relaxation which is to require the condition $\int_{X}\lambda(x) q_i^2(x) d\mu\geq 0$ only on $S(e_1)$-invariant forms $q(x)$. The matrices in this restricted linear program only grow linearly with $k$ and the optimum of the relaxation yields an upper bound on $\alpha(\mathcal{F}_k)$. We compute the optima of these relaxations via a combination of Macaulay2 (used to decompose harmonic polynomials and write down the matrices), YALMIP (used to write down the semidefinite program) and SeDuMi (used to solve it numerically). The results of these computations are written in the statement of the Theorem.
\end{proof}

\begin{proof} [Proof of Theorem~\ref{Thm: PSDbound}] Given $b_j\in \mathbb{R}$ for $j=d+1,\dots, k$ let $q(x_1)=\sum_{j=0}^{d}h_{2j}L_{2j}(x_1)+\sum_{j=d+1}^{k} b_j L_{2j}(x_1)$. Define 
\[Q(x_1,\dots, x_n)=\sum_{j=0}^{d}Z_{2j}(x)s^{k-j}+\sum_{j=d+1}^{k} \frac{b_j}{h_{2j}} Z_{2j}(x)s^{k-j}\]
where $s:=x_1^2+\dots+x_n^2$. The polynomial $Q(x_1,\dots, x_n)$ is homogeneous of degree $2k$ and is such that $Q(u_1,\dots, u_n)=q(u_1)$ for all points $(u_1,\dots, u_n)\in S^{n-1}$. In particular, $Q$ is $S(e_1)$-invatiant and the equality $\min_{x_1\in [-1,1]}q(x_1)=\min_{u\in S^{n-1}} Q(u)$ holds. Letting $\beta:=\min_{u\in S^{n-1}} Q(u)$ we see that $Q(x)-\beta s^{k}$ is a nonnegative homogeneous polynomial of degree $2k$ fixed by the action of $S(e_1)$. By~\cite[Lemma 6.1]{BC} every nonnegative $S(e_1)$-invariant form is a sum of squares and thus $\Lambda(x):=Q(x)-\beta s^k$ is a sum of squares of forms in $\mathcal{F}_k$. 
Now let $\lambda\in B^{\vee}(\mathcal{F}_k)$ be an $S(e_1)$-invariant form. The function $\lambda(\phi(u))$ is the restriction of a unique $S(e_1)$-invariant homogeneous polynomials of degree $2d$ in $n$-variables $\sum_{j=0}^d a_j Z_{2j}(x)s^{d-j}$. We have
\[ 0\leq \int_{S^{n-1}}\lambda(\phi(u))\Lambda(u)d\mu(u)=\int_{S^{n-1}} \left(\sum_{j=0}^d a_j Z_{2j}\right)(Q(x)-\beta s^k)d\mu=\lambda(\phi(e_1))-\beta\] 
where the last equality follows from the orthogonality of harmonic polynomials of different degrees, the defining property of zonal harmonics and from the fact that $\int_{S^{n-1}}\lambda(\phi(u))d\mu(u)=1$. Since $q$ was an arbitrary element of $R_k$ we conclude that 
\[\lambda(\phi(e_1))\geq \sup_{q\in R_k}\min_{x_1\in [-1,1]} q(x_1).\] 
The claimed inequalities follow from Theorem~\ref{Thm: Scaling}.
\end{proof}
\begin{proof}  [Proof of Corollary~\ref{cor: PSD1}]  By~\cite[Answer 2]{MathOverFlow} we know that the Legendre polynomial $L_j(x_1)$ satisfies the following inequalities for all $x_1\in [-1,1]$
\[ -\frac{1}{(1-x_1^2)^{\frac{1}{4}}}\sqrt{\frac{4}{\pi(2j+1)}}\leq L_j(x_1)\leq \frac{1}{(1-x_1^2)^{\frac{1}{4}}}\sqrt{\frac{4}{\pi(2j+1)}}.\]
Bounding $q(x_1):=\sum_{j=0}^{d} h_{2j}L_{2j}(x_1)$ term by term we obtain the lower bound 
\[r(x_1):=-\frac{1}{(1-x_1^2)^{\frac{1}{4}}}\left( \sum_{j=0}^d h_{2j} \sqrt{\frac{4}{\pi(4j+1)}}\right)\] for $q(x_1)$. The function $r(x_1)$ is even and decreasing in $[0,1]$. Since $q(1)>0$, the absolute minimum of $q(x_1)$ must be achieved at a point $x_1\in [0,1]$ smaller than the biggest root of $q(x_1)$. Since the roots of the Legendre polynomials interlace we know that the largest root of $q(x_1)$ must be smaller than the largest root $\gamma$ of $L_{2d}(x_1)$. Since $r(x_1)$ is decreasing in $[0,1]$ we obtain a lower bound by for $q(x_1)$ by evaluating $r(x_1)$ at $\gamma$ as claimed. The given asymptotic formula for the largest roots is due to Gatteschi~\cite{Gat}.
\end{proof}

\begin{figure}[h]
\includegraphics[scale=0.5]{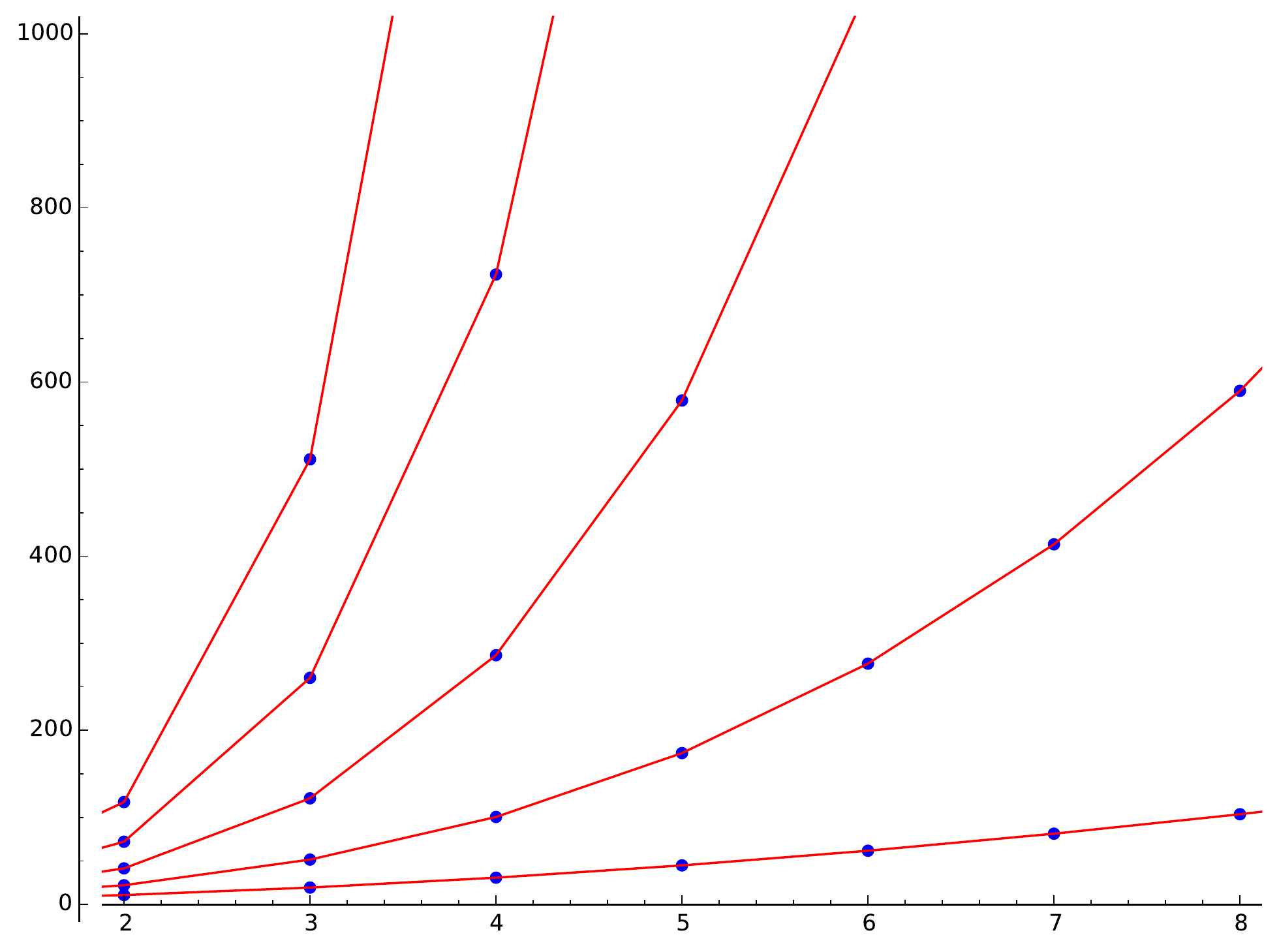}
\caption{Upper bound for the scaling constant $\alpha(\mathcal{F}_k)$ for $n=3,\dots,7$ and $k=d=2,\dots, 8$.}
\end{figure}

\begin{remark} \label{Heuristics} For any even harmonic polynomial $f$ of degree at most $d$ we know $\int_{S^{n-1}} f(u)\sum_{j=0}^{\infty} Z_{2j}(u) d\mu(u)=f(e_1)$ so that  $\sum_{j=0}^{\infty} Z_{2j}$ is the Fourier transform of the distribution $\frac{\delta_{e_1}+\delta_{-e_1}}{2}$. However the sequence of partial sums $\eta_m(x)=\sum_{j=0}^{m} Z_{2j}$ does not converge in any sense and the minimim value keeps decreasing. We conjecture (see Figure~\ref{FourierDelta}) that the minimum of this partial sum is always achieved at the largest zero of its derivative, leading to an improvement on the bound of Corollary~\ref{cor: PSD1}.

\begin{figure}[h]

\includegraphics[scale=0.5]{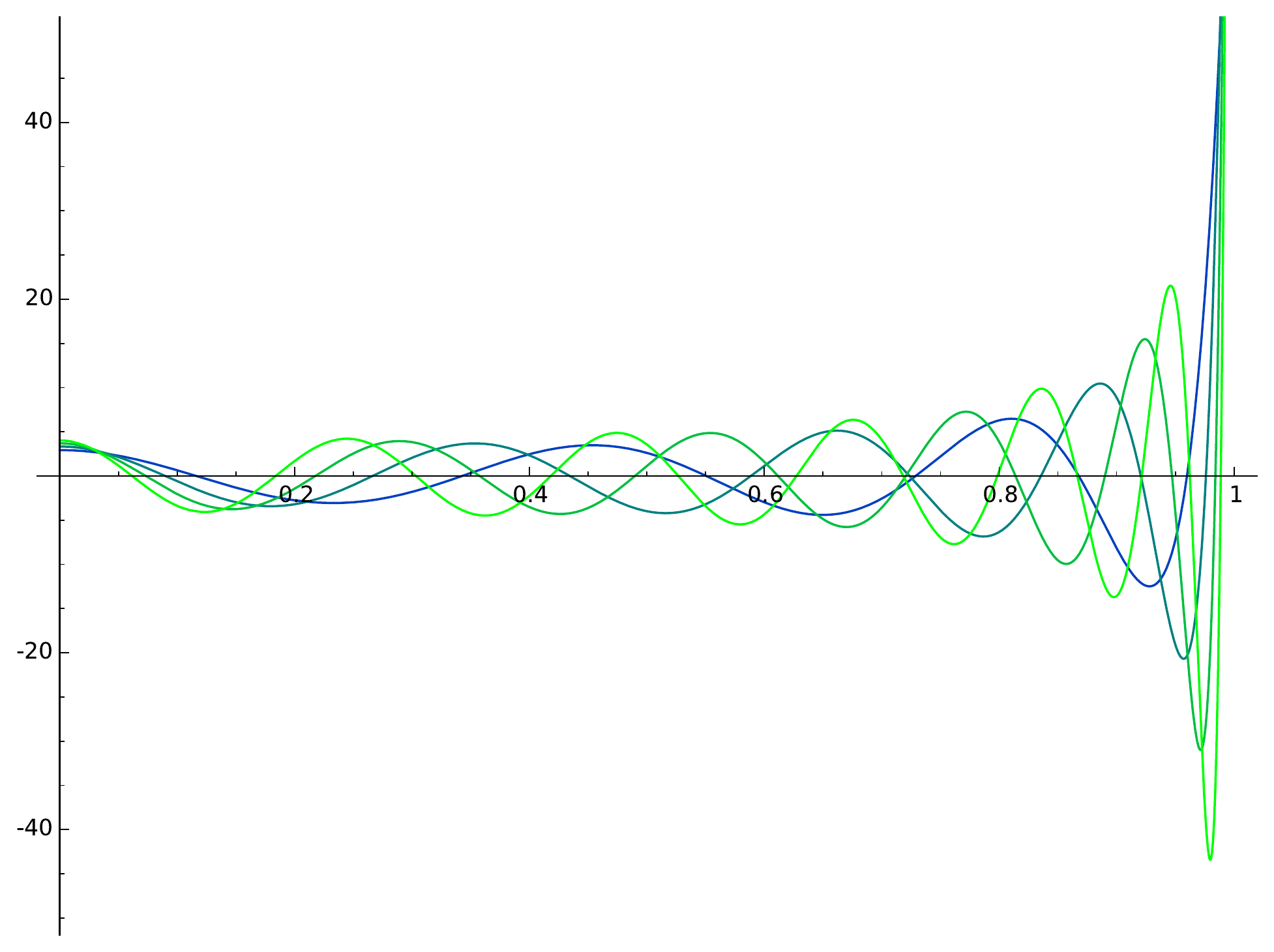}\label{FourierDelta}
\caption{Polynomials $\eta_m(x)$ for $n=3$ and $m=6,8,10,12$.}
\end{figure}
\end{remark}

\begin{remark} Work in preparation by P. Parrilo gives optimal approximations of the Dirac delta by sums of squares~\cite{PP}. Such expressions could be used to obtain potentially sharp upper bounds for scaling constants of BVL approximations of the cone of nonnegative polynomials.
\end{remark}

\end{document}